\documentclass[a4paper,UKenglish,cleveref, autoref, thm-restate]{lipics-v2021}

\usepackage{tabularx}
\usepackage{environ}
\usepackage{lipsum}
\usepackage{amsfonts}
\usepackage{cite}
\usepackage{xspace}
\usepackage{array}
\usepackage{xparse}
\usepackage{amsmath}
\usepackage{complexity}
\usepackage{url}[spaces, hyphens]
\usepackage{ulem}
\usepackage[scr=kp]{mathalfa} 

\usepackage{tikz}

\usetikzlibrary{mindmap,positioning}
\usetikzlibrary{graphs}
\usetikzlibrary[graphs]
\usetikzlibrary{patterns}
\usetikzlibrary{arrows,automata,positioning}
\usetikzlibrary{decorations.pathreplacing}
\usetikzlibrary{decorations}
\usetikzlibrary{decorations.pathmorphing}
\usetikzlibrary{shapes.geometric}
\usetikzlibrary{quotes}
\usetikzlibrary{shapes.callouts}
\usetikzlibrary{arrows.meta}
\usetikzlibrary{calc}
\usetikzlibrary{fit}

\tikzset{faded/.style={gray,very thin}}
\tikzset{vertex/.style={draw,circle,minimum size=5pt,inner sep=0pt}}
\tikzset{novertex/.style={circle,minimum size=5pt,inner sep=0pt}}
\tikzset{blackvertex/.style={draw,circle,minimum size=5pt,inner sep=0pt, fill=black}}
\tikzset{redvertex/.style={draw,circle,minimum size=5pt,inner sep=0pt, fill=red}}
\tikzset{redvertexfaded/.style={draw,circle,faded,minimum size=5pt,inner sep=0pt, fill=red!50}}
\tikzset{greenvertex/.style={draw,circle,minimum size=5pt,inner sep=0pt, fill=green}}
\tikzset{greenvertexfaded/.style={draw,circle,faded,minimum size=5pt,inner sep=0pt, fill=green!50}}
\tikzset{bluevertex/.style={draw,circle,minimum size=5pt,inner sep=0pt, fill=blue}}
\tikzset{bluevertexfaded/.style={draw,circle,faded,minimum size=5pt,inner sep=0pt, fill=blue!50}}
\tikzset{yellowvertex/.style={draw,circle,minimum size=5pt,inner sep=0pt, fill=yellow}}
\tikzset{yellowvertexfaded/.style={draw,circle,faded,minimum size=5pt,inner sep=0pt, fill=yellow!50}}
\tikzset{arrow/.style={-{Latex[scale=1]},shorten >= 4pt}}

\tikzset{snake it/.style={decorate, decoration=snake}}
\tikzset{big snake/.style={decorate, decoration={snake,segment length=7mm, amplitude=3mm}}}

\newcommand{\lipItem}[1]{\textcolor{lipicsGray}{\sffamily\bfseries\upshape\mathversion{bold}#1}}

\pdfoutput=1 
\hideLIPIcs  

\title{New Menger-like dualities in digraphs and applications to half-integral linkages} 

\titlerunning{New Menger-like dualities in digraphs and applications to half-integral linkages}

\author{Victor Campos}{ParGO group, Universidade Federal do Ceará, Fortaleza, Brazil}{victoitor@ufc.br}{https://orcid.org/0000-0002-2730-4640}{FUNCAP/PRONEM PNE-0112-00061.01.00/16, and CAPES/Cofecub 88887.712023/2022-00.}

\author{Jonas Costa}{ParGO group, Universidade Federal do Ceará, Fortaleza, Brazil}{jonascosta@lia.ufc.br}{}{FUNCAP/PRONEM PNE-0112-00061.01.00/16, FUNCAP PS1-0186-00155.01.00/2, and PhD scholarship granted by CAPES.}

\author{Raul Lopes}{DIENS, École normale supérieure, CNRS, PSL University, Paris, France \and Université Paris-Dauphine, CNRS UMR7243, PSL University, Paris, France}{raul.wayne@gmail.com}{https://orcid.org/0000-0002-7487-3475}{group Casino/ENS Chair on Algorithmics and Machine Learning, and French National Research Agency under JCJC program (ASSK: ANR-18-CE40-0025-01).}

\author{Ignasi Sau}{LIRMM, Université de Montpellier, CNRS, Montpellier, France}{ignasi.sau@lirmm.fr}{https://orcid.org/0000-0002-8981-9287}{ELIT (ANR-20-CE48-0008-01).}

\authorrunning{ }

\Copyright{ }

\ccsdesc[100]{Mathematics of computing $\to$ Discrete mathematics $\to$ Graph theory $\to$ Graph algorithms.} 

\keywords{directed graphs, min-max relation, half-integral linkage, directed disjoint paths, bramble, parameterized complexity, matroids.} 

\category{} 

\relatedversion{A conference version of this article will appear in the \emph{Proceedings of the 31st Annual European Symposium on Algorithms (\textbf{ESA}), LIPIcs, Amsterdam, The Netherlands, September \textbf{2023}}.} 

\nolinenumbers 

\EventEditors{Inge Li Gørtz,
Simon J. Puglisi, and
Martin Farach-Colton,}
\EventNoEds{2}
\EventLongTitle{ESA 2023?}
\EventShortTitle{ESA 2023?}
\EventAcronym{ESA}
\EventYear{2023}
\EventDate{September 4--6, 2023}
\EventLocation{Amsterdam, The Netherlands}
\EventLogo{}
\SeriesVolume{XXX}
\ArticleNo{ }

\DeclareMathOperator{\source}{\text{{\sf s}}^-\xspace}
\DeclareMathOperator{\sink}{\text{{\sf s}}^+\xspace}
\DeclareMathOperator{\order}{\text{\sf ord}\xspace}
\DeclareMathOperator{\dtw}{\text{\sf dtw}\xspace}
\DeclareMathOperator{\bn}{\text{\sf bn}\xspace}

\newcommand{\Dpaths}{\textsf{D}-paths\xspace}
\newcommand{\Tpaths}{\textsf{T}-paths\xspace}
\newcommand{\Rpaths}{\textsf{R}-paths\xspace}

\newcommand{\Dcut}{\textsf{D}-cut\xspace}
\newcommand{\Tcut}{\textsf{T}-cut\xspace}
\newcommand{\Rcut}{\textsf{R}-cut\xspace}

\newcommand{\Ocal}{\mathcal{O}\xspace}

\makeatletter

\newcolumntype{\expand}{}
\long\@namedef{NC@rewrite@\string\expand}{\expandafter\NC@find}

\NewEnviron{boxproblem}[2][]{%
  \def\boxproblem@arg{#1}%
  \def\boxproblem@framed{framed}%
  \def\boxproblem@lined{lined}%
  \def\boxproblem@doublelined{doublelined}%
  \ifx\boxproblem@arg\@empty%
    \def\boxproblem@hline{}%
  \else%
    \ifx\boxproblem@arg\boxproblem@doublelined%
      \def\boxproblem@hline{\hline\hline}%
    \else%
      \def\boxproblem@hline{\hline}%
    \fi%
  \fi%
  \ifx\boxproblem@arg\boxproblem@framed%
    \def\boxproblem@tablelayout{|>{\bfseries}lX|c}%
    \def\boxproblem@title{\multicolumn{2}{|l|}{%
        \raisebox{-\fboxsep}{\textsc{\normalsize #2}}%
      }}%
  \else
    \def\boxproblem@tablelayout{>{\bfseries}lXc}%
    \def\boxproblem@title{\multicolumn{2}{l}{%
        \raisebox{-\fboxsep}{\textsc{\normalsize #2}}%
      }}%
  \fi%
  \bigskip\par\noindent%
  \begin{tabularx}{\textwidth}{\expand\boxproblem@tablelayout}%
    \boxproblem@hline%
    \boxproblem@title\\[2\fboxsep]%
    \BODY\\\boxproblem@hline%
  \end{tabularx}%
  \medskip\par%
}
\makeatother

\newcommand{\cupall}{\pmb{\pmb{\bigcup}}}

\definecolor{mid-green}{rgb}{0.15,0.65,0.15}
\definecolor{dark-green}{rgb}{0.15,0.25,0.15}
\definecolor{dark-red}{rgb}{0.7,0.15,0.15}
\definecolor{dark-blue}{rgb}{0.15,0.15,0.9}
\definecolor{medium-blue}{rgb}{0,0,0.5}
\definecolor{gray}{rgb}{0.5,0.5,0.5}
\definecolor{color-Ig}{rgb}{0.15,0.7,0.15}
\definecolor{darkmagenta}{rgb}{0.30, 0.0, 0.30}

\newcommand{\red}{\textcolor{red}}

\hypersetup{
    colorlinks, linkcolor={dark-blue},
    citecolor={mid-green}, urlcolor={dark-blue}}

\renewcommand{\NP}{{\sf NP}\xspace}
\renewcommand{\P}{{\sf P}\xspace}
\renewcommand{\FPT}{{\sf FPT}\xspace}
\renewcommand{\XP}{{\sf XP}\xspace}

\begin{document}

\maketitle

\begin{abstract}
We present new min-max relations in digraphs between the number of paths satisfying certain conditions and the order of the corresponding cuts. We define these objects in order to capture, in the context of solving the half-integral linkage problem, the essential properties  needed for reaching a large bramble of congestion two (or any other constant) from the terminal set. This strategy has been used ad-hoc in several articles, usually with lengthy technical proofs, and our objective is to abstract it to make it applicable in a simpler and unified way. We provide two proofs of the min-max relations, one consisting in applying Menger's Theorem on appropriately defined auxiliary digraphs, and an alternative simpler one using matroids, however with worse polynomial running time.

As an application, we manage to simplify and improve several results of Edwards et al.~[ESA 2017] and of Giannopoulou et al.~[SODA 2022] about finding half-integral linkages in digraphs. Concerning the former, besides being simpler, our proof provides an almost optimal bound on the strong connectivity of a digraph for it to be half-integrally feasible under the presence of a large bramble of congestion two (or equivalently, if the directed tree-width is large, which is the hard case). Concerning the latter, our proof uses brambles as rerouting objects instead of cylindrical grids, hence yielding much better bounds and being somehow independent of a particular topology.

We hope that our min-max relations will find further applications as, in our opinion, they are simple, robust, and versatile to be easily applicable to different types of routing problems in digraphs.
\end{abstract}


\newpage
\setcounter{page}{1}
\section{Introduction}\label{sec:intro}
In combinatorial optimization, a \emph{min-max relation} establishes the equality between two quantities, one naturally associated with {\sl minimizing} the size of an object satisfying some conditions, and the other one associated with {\sl maximizing} the size of another object. Within graph theory, famous such min-max relations include Kőnig's Theorem~\cite{Konig31} stating the equality between the sizes of a maximum matching and a minimum vertex cover in a bipartite graph or, more relevant to this article, Menger's Theorem~\cite{Menger1927} stating, in its simplest form, the equality between the maximum number of pairwise internally disjoint paths between two vertices, and the minimum size of a vertex set disconnecting them. Typically, min-max relations come along with polynomial-time algorithms to find the corresponding objects, making them extremely useful from the algorithmic point of view.

In this article we focus on directed graphs, or \emph{digraphs} for short, and our results are motivated by the complexity of problems related to finding \emph{directed disjoint paths} between given terminals. More precisely, in the $k$-\textsc{Directed Disjoint Paths} problem ($k$-\textsc{DDP} for short), we are given a digraph $G$ and $k$ pairs of vertices $s_i,t_i, i \in [k]$, and the objective is to decide whether $G$ contains $k$ pairwise disjoint paths connecting $s_i$ to $t_i$ for $i \in [k]$. A solution to this problem is usually called a \emph{linkage} in the literature.
Here we note that disjoint paths is equivalent to paths which are vertex-disjoint.

Unfortunately, Fortune et al.~\cite{FortuneHW80} proved that the $k$-\textsc{DDP} problem is \NP-complete already for $k=2$, and Thomassen~\cite{Thomassen91} strengthened this result by showing that it remains so even if the input digraph is $p$-strongly connected (see \autoref{sec:preliminaries} for the definition) for any integer $p \geq 1$. Thus, in order to obtain positive algorithmic results, research has focused on either restricting the input digraphs (for instance, to being acyclic~\cite{Slivkins2010} or, more generally, to having bounded directed tree-width~\cite{Johnson2001}), or on considering relaxations of the problem. Concerning the latter, a natural candidate is to relax the disjointness condition of the paths, and allow for \emph{congestion} in the vertices. Namely, for an integer $c \geq 2$, an input of the $k$-\textsc{Directed $c$-Congested Disjoint Paths} problem ($(k,c)$-\textsc{DDP} for short) is the same as in the $k$-\textsc{DDP}, but now we allow each vertex of $G$ to occur in at most $c$ out of the $k$ paths connecting the terminals. In the particular case $c=2$, a solution to this problem is usually called a \emph{half-integral linkage} in the literature.

Despite a considerable number of attempts, it is still open whether the $(k,c)$-\textsc{DDP} problem can be solved in polynomial time for every fixed value of $c \geq 2$ and $k > c$ (note that if $k \leq c$, then the problem can be easily solvable in polynomial time just by verifying the connectivity between each pair of terminals). A positive answer for the case $c=2$  has been recently conjectured by Giannopoulou et al.~\cite{GiannopoulouKKK22}. Again, in order to obtain positive results, several restrictions and variations of the problem have been considered, such as considering several parameterizations~\cite{AmiriKMR19,Lopes2022}, restricting the input graph to have high connectivity~\cite{Edwards2017}, or considering an \emph{asymmetric} version of the $(k,c)$-\textsc{DDP} problem~\cite{KawarabayashiK15,KawarabayashiKK14,GiannopoulouKKK22}, where the input is as in $(k,c)$-\textsc{DDP}, but the goal is to either certify that it is a {\sf no}-instance of $k$-\textsc{DDP} (without congestion) or a {\sf yes}-instance of $(k,c)$-\textsc{DDP}. This asymmetric version has been solved in polynomial time for every fixed $k$ (i.e., showing that it is in \XP; see \autoref{sec:preliminaries}) for distinct values of $c$ in a series of articles, namely for $c=4$ by Kawarabayashi et al.~\cite{KawarabayashiKK14}, for $c=3$ by Kawarabayashi and Kreutzer~\cite{KawarabayashiK15}, and for $c=2$ by Giannopoulou et al.~\cite{GiannopoulouKKK22}.

The main motivation of this article stems from the techniques used in the latter two approaches mentioned above. In a nutshell, the main strategy used in~\cite{Edwards2017,KawarabayashiKK14,KawarabayashiK15,GiannopoulouKKK22} is the following. First, one computes whether the directed tree-width of the input graph is bounded by an appropriate function of $k$, the number of terminal pairs. This can be done in time \XP in $k$ by the results of
Johnson et al.~\cite{Johnson2001}, or even in time \FPT by the results of Campos et al.~\cite{Campos2022} (see also \cite[Theorem 9.4.4]{BangJensen2018}).
If the directed tree-width is bounded by a function of $k$, one solves the problem in time \XP by using standard programming techniques from Johnson et al.~\cite{Johnson2001} (cf. \autoref{proposition:XP-algo-DDP-congestion}). If not, one exploits the fact that large directed tree-width implies the existence of large ``structures'' that can be used to carry out the routing of the desired paths. Typically, such a structure is a \emph{bramble}, as for example in~\cite{Edwards2017}, or a \emph{cylindrical grid}, as for example in~\cite{GiannopoulouKKK22,KawarabayashiK15}, making use of the celebrated Directed Grid Theorem of Kawarabayashi and Kreutzer~\cite{KawarabayashiK15} (see also~\cite{Campos2022} for recent improvements). For the sake of exposition, assume henceforth that the desired structure is a bramble, but the strategy is essentially the same with a cylindrical grid.

A bramble in a digraph $D$ is a set $\mathcal{B}$ of strongly connected subgraphs of $G$ that pairwise either intersect or have edges in both directions. The \emph{order} of a bramble $\mathcal{B}$ is the smallest size of a vertex set of $G$ that intersects all its elements, and its \emph{congestion} is the maximum number of times that a vertex of $G$ appears in the elements of $\mathcal{B}$. It is known that large directed treewidth implies the existence of a bramble of large order and of congestion $c \geq 2$.
For $c=2$, a proof about how to find such a bramble in polynomial time (with degree not depending on $k$), provided that a certificate for large directed tree-width is given, can be found in~\cite{Edwards2017} (see also~\cite{MasarikPRS22} for improved bounds for brambles of higher congestion). Assume for simplicity that $c=2$, let $\mathcal{B}$ be such a bramble, and let $S$ and $T$ be the sets of sources and sinks, respectively, of the corresponding problem. The idea is that if one can find a set $\mathscr{P}^S$ of disjoint paths from $S$ to appropriate elements of $\mathcal{B}$, and a set $\mathscr{P}^T$ of disjoint paths from appropriate elements of $\mathcal{B}$ to $T$ (regardless of the ordering of the vertices of $S$ and $T$), then we are done. Indeed, once the paths starting in $S$ reach $\mathcal{B}$, one can use the connectivity properties of the bramble to ``shuffle'' the paths appropriately as required by the terminals, and then follow the paths from $\mathcal{B}$ to $T$. The fact that the bramble has congestion two, and that the paths in $\mathscr{P}^S$, as well as those in $\mathscr{P}^T$, are pairwise disjoint, together with a good choice for the destination and starting points of the those paths, implies that every vertex of $G$ occurs in at most two of the resulting paths.

Otherwise, if such sets of paths $\mathscr{P}^S$ and $\mathscr{P}^T$ do {\sl not} exist, the approach consists in using a Menger-like min-max duality to obtain an appropriate \emph{separator} (or \emph{cut}) between the terminals and the bramble of size bounded by a function of $c$ and $k$, and make some progress toward the resolution of the problem, for instance by splitting into subproblems of lower complexity. The ways to define and to exploit such a separator depend on every particular application, and this ad-hoc subroutine is usually one of the most technically involved parts of the resulting algorithms~\cite{Edwards2017,KawarabayashiKK14,KawarabayashiK15,GiannopoulouKKK22}.

\medskip
\noindent\textbf{Our results and techniques}. Motivated by the inherent common essential strategy in the above articles, we aim at finding the crucial general ingredient that can be applied in order to define and find the corresponding separators. To this end, we introduce new objects that abstract the existence of the aforementioned desired paths $\mathscr{P}^S$ and $\mathscr{P}^T$ between the terminals and the bramble.
These objects are what we call \emph{\Dpaths}, \emph{\Tpaths} and \emph{\Rpaths}.
The inspiration for \Dpaths and \Tpaths is what we believe to be the common essential strategies used by Edwards et al.~\cite{Edwards2017} and Giannopoulou et al.~\cite{Giannopoulou2020}.
We remark that a particular set of \Tpaths is directly constructed in~\cite{Giannopoulou2020}, particularly inside the proof of~\cite[Theorem 9.1]{Giannopoulou2020}.
The presence of \Dpaths and \Tpaths in~\cite{Edwards2017} is more subtle in the construction of a long algorithm and a collection of non-trivial proofs.
In fact, they are not explicitly built due to constraints in their techniques, but our initial results for this paper included simplifying and improving the proofs in~\cite{Edwards2017} using \Dpaths and \Tpaths.
Once the new proofs were obtained, we noticed that they could be further improved by a new object which we call \Rpaths.
It must be noted that this paper includes results for \Dpaths, \Tpaths, and \Rpaths but we only show applications for \Rpaths.
The reason for this is twofold.
First, \Tpaths have stronger properties than \Rpaths which might be useful for solving different problems, and \Dpaths are needed in the proof of the duality of \Tpaths.
And second, these three objects are similar to the point of having similar proofs for their min-max formulas and algorithms, and showing these variations incurs little additional effort. 
The formal definitions of these special types of paths can be found in \autoref{sec:newpaths}.

All three types of paths are associated with a defining partition $\mathcal{P}_1, \ldots, \mathcal{P}_{\ell}$ sharing some properties and differing in others.
For all three, it holds that any two paths in a same part $\mathcal{P}_i$ are disjoint, and it is possible that two paths in distinct parts share vertices.
In the context of the $(k,c)$-\textsc{DDP} problem, the main difference between \Dpaths, \Tpaths, and \Rpaths lies in how we want to reach (or, by applying a simple trick of reversing the edges of the digraph, be reached from) the elements of a bramble $\mathcal{B}$ in the given digraph.
In \Dpaths we ask that all paths end in distinct vertices of a given $B \subseteq V(G)$.
In \Tpaths we ask that all paths end in distinct vertices of the bramble, and that the set containing all last vertices of the \Tpaths forms a \emph{partial transversal} (see \autoref{sec:preliminaries} for the definition) of $\mathcal{B}$.
In \Rpaths we ask that all \Rpaths end in elements of $\mathcal{B}$ and that there is a ``matching-like'' association between the last vertices of the paths and elements of $\mathcal{B}$ containing these vertices.

For each type of paths we define an associated notion of \emph{cut} and its corresponding \emph{order} (see e.g. \autoref{def:D-paths-and-D-cuts}).
These cuts are respectively called \emph{\Dcut{s}}, \emph{\Tcut{s}}, and \emph{\Rcut{s}}.
We show that each of these special types of paths and cuts satisfies a Menger-like min-max duality, that is, that the maximum number of paths equals the minimum order of a cut (cf.~\cref{theo:min-max-statement-var-1,theo:min-max-statement-var-2,theo:min-max-statement-var-3}).
Moreover, the corresponding objects attaining the equality  can be found in polynomial time. The proofs of these min-max relations basically consist in applying Menger's Theorem~\cite{Menger1927} in appropriately defined auxiliary graphs.
In \red{\autoref{subsection:proofs-by-matroids}} we provide alternative simpler proofs of these equalities using intersections and unions of matroids, namely gammoids and transversal matroids. Even if the resulting polynomial-time algorithms using matroids have worse running time than the ones that we obtain by applying Menger's Theorem~\cite{Menger1927}, and that using the deep theory of matroids somehow sheds less light on interpreting the actual behavior of the considered objects, we think that it is interesting to observe that the paths and cuts that we define are in fact matroids with nice properties.

\smallskip

The main application we provide for \Rpaths/cuts in this article is in the proof of \autoref{theorem:solution_or_small_separator}, which is an improved version of~\cite[Theorem 9.1]{Giannopoulou2020} both in the requested order of the structure and by relying on brambles instead of cylindrical grids.
Informally, \autoref{theorem:solution_or_small_separator} says that given a digraph $G$, ordered sets $S,T \subseteq V(G)$, and a large (depending on the congestion $c$ and on the size $k$ of $S$ and $T$) bramble of congestion $c$, we can either find a large set of \Rpaths from $S$ to the bramble and from the bramble to $T$, which in turn are used to appropriately connect the pairs $s_i \in S, t_i \in T$, or find a separator of size at most $k-1$ intersecting every path from $S$ to a large subset of the bramble, or every path from a large subset of the bramble to $T$.
Additionally, if the bramble is given, one of the outputs can be obtained in polynomial time, computing either the paths or one of the separators.

Since in $k$-strong digraphs the separators are never found, \autoref{theorem:solution_or_small_separator} immediately implies an improved version of a result by Edwards et al.~\cite{Edwards2017}.
Namely, in~\cite[Theorem 11]{Edwards2017} the authors show that, when restricted to $(36k^3 + 2k)$-strong digraphs, every instance of $(k,2)$-\textsc{DDP} where the input digraph contains a large (depending on $k$) bramble of congestion two is positive and a solution can be found in polynomial time.
When compared to theirs, our result is an improvement in the following ways.
First, it allows us to solve $(k,c)$-\textsc{DDP} for any $c \geq 2$ in the larger class of $k$-strong digraphs instead of being restricted to $(36k^3 + 2k)$-strong digraphs as in~\cite{Edwards2017}.
This bound on the strong connectivity of the digraph is almost best possible according to~\cite[Theorem 2]{Edwards2017}, unless $\P = \NP$ (note that an {\XP} algorithm in $(k-d)$-strong digraphs, for some constant $d$, may still be possible).
Second, we show how to find the desired paths using a bramble of congestion $c$ and size at least $2k(c\cdot k - c + 2) + c(k-1)$, which is equal to $4k^2 + 2k - 2$ when $c=2$, instead of the size $188k^3$ required in~\cite{Edwards2017}.
Finally, our proof is much simpler and shorter than the proof presented in~\cite{Edwards2017}.
A main reason of this simplification is that we can replace the seven properties of the paths requested in~\cite[Lemma 12]{Edwards2017} by \Rpaths. It is worth mentioning that our algorithm reuses the procedure of Edwards et al.~\cite{Edwards2017} to find a large bramble of congestion two in digraphs of large directed tree-width (cf. \autoref{cor:bounded-dtw-or-bramble-congestion-two}).

We remark that it is also possible to improve the result by Edwards et al.~\cite{Edwards2017} from $(36k^3 + 2k)$-strong digraphs to $k$-strong digraphs by replacing part of their proof, namely~\cite[Theorem 11]{Edwards2017}, by the result of Giannopoulou et al.~\cite[Theorem 9.1]{Giannopoulou2020}.
The trade-off is that the latter relies on the stronger structure of a cylindrical grid (and such grids do contain brambles of congestion two~\cite{Edwards2017}) instead of brambles.
A fundamental difference stands on the fact that, given a certificate of large directed tree-width, one can produce a bramble of congestion two in polynomial time, while finding a cylindrical grid still requires {\FPT} time parameterized by the order of the certificate~\cite{Campos2022}.
Our result using \Rpaths keeps the best of both worlds: we are able to drop the request on the strong connectivity of the digraph to $k$ while relying only on brambles as the routing structures.

\smallskip

Our second application deals with the asymmetric version of the $(k,c)$-\textsc{DDP} problem discussed in the introduction.
By using our min-max relations, we manage to simplify and improve one of the main results of Giannopoulou et al.~\cite{GiannopoulouKKK22} for the case $c=2$.
Instead of using the Directed Grid Theorem~\cite{KawarabayashiK15} to reroute the paths through a cylindrical grid, we reroute them through a bramble of congestion two in a very easy manner after a careful choice of the paths reaching and leaving the bramble, which is done by applying the duality between \Rpaths and \Rcut{s}.
Namely, we can replace~\cite[Theorem 9.1]{Giannopoulou2020} (this is the full version of~\cite{GiannopoulouKKK22}) entirely by \autoref{theorem:solution_or_small_separator} and mostly keep the remaining part of their proof to obtain an improved version of their {\XP} algorithm for the asymmetric version of $(k,2)$-\textsc{DDP}.

\smallskip

We hope that our min-max relations will find further applications in the future as, in our opinion, they are quite simple, robust, and versatile to be easily applicable to different types of routing problems in digraphs. A natural candidate is the $(k,c)$-\textsc{DDP} problem for any choice of fixed values of $c \geq 2$ and $k > c$, which has remained elusive for some time.

\medskip
\noindent\textbf{Organization}. In \autoref{sec:preliminaries} we present some  preliminaries about digraphs, matroids, parameterized complexity, and known results about brambles and related objects. In \autoref{sec:newpaths} we present and prove our new Menger-like statements for paths in digraphs. The alternative proofs using matroids can be found in \autoref{subsection:proofs-by-matroids}, and the proofs using Menger's Theorem in \autoref{subsection:proofs-by-menger}. The applications of our results can be found in \autoref{section:DDP-algorithm}.

\section{Preliminaries}\label{sec:preliminaries}

In this section we provide the definitions and notation used in this text.
If $\mathcal{B}$ is a collection of sets, we denote the we use $\cupall \mathcal{B}$ to denote the set $\bigcup_{A \in \mathcal{B}}A$.
For a positive integer $k$, we denote by $[k]$ the set $\{1, \ldots, k\}$.

\subsection{Digraphs}
For basic background on graph theory we refer the reader to~\cite{Bondy2008}.
Since in this article we mainly work with digraphs, we focus on basic definitions of digraphs, often skipping their undirected counterparts.
Unless stated otherwise, every edge is assumed to be directed, and given a digraph $G$ we denote by $V(G)$ and $E(G)$ the sets of vertices and edges of $G$, respectively.
We allow for loops and parallel edges in our digraphs.
Given $X \subseteq V(G)$, we denote by $G \setminus X$ the digraph resulting from removing every vertex of $X$ from $G$.
We denote by $G[X]$ the subgraph of $G$ \emph{induced} by $X$.

If $e$ is an edge of a digraph from a vertex $u$ to a vertex $v$, we say that $e$ has \emph{endpoints} $u$ and $v$, that $e$ is \emph{incident} to $u$ and $v$, and that $e$ is \emph{oriented} from $u$ to $v$. We may refer to $e$ as the ordered pair $(u,v)$.
In this case, $u$ is the \emph{tail} of $e$ and $v$ is the \emph{head} of $e$.
We also say that $e$ is \emph{leaving} $u$ and \emph{reaching} $v$, and that $u$ and $v$ are \emph{adjacent}.
A \emph{clique} in a (di)graph $G$ is a set of pairwise adjacent vertices of $G$, and an \emph{independent set} is a set of pairwise non-adjacent vertices of $G$.

The \emph{in-degree} (resp. \emph{out-degree}) of a vertex $v$ in a digraph $G$ is the number of edges with head (resp. tail) $v$.
The \emph{in-neighborhood} $N^-_D(v)$ of $v$ is the set $\{u \in V(D) \mid (u,v) \in E(G)\}$, and the \emph{out-neighborhood} $N^+_D(v)$ is the set $\{u \in V(D) \mid (v,u) \in E(G)\}$.
We say that $u$ is an \emph{in-neighbor} of $v$ if $u \in N^-_D(v)$ and that $u$ is an \emph{out-neighbor} of $v$ if $u \in N^+_D(v)$.
We extend these notations to sets of vertices: given $X \subseteq V(D)$, we define $N^-_D(X) = (\bigcup_{v \in X}N^-_D(v)) \setminus X$ and $N^+_D(X) = (\bigcup_{v \in X}N^+_D(v)) \setminus X$.

 A \emph{walk} in a digraph $G$ is an alternating sequence $W$ of vertices and edges that starts and ends with a vertex, and such that for every edge $(u,v)$ in the walk, vertex $u$ (resp. vertex $v$) is the element right before (resp. right after) edge $(u,v)$ in $W$.
 If the first vertex in a walk is $u$ and the last one is $v$, then we say this is a \emph{walk from $u$ to $v$}.
A \emph{path} is a digraph formed by pairwise distinct vertices $V(P) = \{v_1, \ldots, v_k\}$ and $E(P) = \{(v_i, v_{i+1}) \mid i \in [k-1]\}$ for some positive integer $k$. 
We say that $v_1$ is the \emph{first} vertex of $P$, $v_k$ is the \emph{last} vertex of $P$ and $P$ is a path \emph{from} $v_1$ \emph{to} $v_k$.
We denote by \emph{$\source(P)$} and \emph{$\sink(P)$} the first and last vertices of $P$, respectively.
Every vertex of $P$ other than $\source(P)$ and $\sink(P)$ is an \emph{internal} vertex.
For $A, B \in V(G)$, we say that $P$ is an \emph{$A \to B$ path} if  $\source(P)\in A$ and $\sink(P)\in B$.
For $A,B \subseteq V(G)$ an \emph{$(A,B)$-separator} is a set $X \subseteq V(G)$ such that there are no $A \to B$ paths in $G \setminus X$.

Let $\mathcal{P}$ be a collection of paths in $G$.
We use $\source(\mathcal{P})$ to denote $\bigcup_{P\in \mathcal{P}}\source(P)$ and $\sink(\mathcal{P})$ to denote $\bigcup_{P\in \mathcal{P}}\sink(P)$.
For conciseness, we say henceforth that the paths in $\mathcal{P}$ are \emph{disjoint} if they are pairwise vertex-disjoint.
For $A, B \in V(G)$, we say that $\mathcal{P}$ is a collection of $A \to B$ paths if  $\source(\mathcal{P}) \subseteq A$ and $\sink(\mathcal{P}) \subseteq B$.
For the remaining of this article and unless stated otherwise, $n$ is used to denote the number of vertices of the input digraph of the problem under consideration.
\begin{theorem}[{Menger's Theorem}~\cite{Menger1927}] \label{thm:Menger}
Let $G$ be a digraph and $A,B \subseteq V(D)$.
The maximum size of a collection of disjoint $A \to B$ paths is equal to the minimum size of an $(A,B)$-separator.
Furthermore, a maximum size collection of paths and a minimum size separator can be found in time $\Ocal(n^2)$. 
\end{theorem}

A digraph $G$ is \emph{strongly connected} if for every $u,v \in V(G)$ there are paths from $u$ to $v$ and from $v$ to $u$ in $G$.
A \emph{separator} of $G$ is a set $X \subseteq G$ such that $G \setminus X$ has a single vertex or is \textsl{not} strongly connected.
If $G$ has at least $k+1$ vertices and $k$ is the minimum size of a separator of $G$, we say that $G$ is \emph{$k$-strongly connected} (or \emph{$k$-strong} for short).
A \emph{strongly connected component} (or \emph{strong component} for short) of a digraph $G$ is a maximal induced subgraph of $G$ that is strongly connected.

\subsection{Matroids}

We refer the reader to~\cite{Oxley2011} for background on matroid theory and to \cite{Schrijver2002} for a compendium of classical results of the field.
For a matroid $M$, we denote its ground set by $E(M)$ and its collection of independent sets by $\mathcal{I}(M)$.
We say that $M$ is a matroid \emph{on} $E$ if the ground set of $M$ is $E$.
A subset $I$ of $E(M)$ is \emph{independent} if $I \in \mathcal{I}(M)$.
The \emph{rank} $r_M(U)$ of $U \subseteq E(M)$ is the largest size of an independent subset of $U$.

For a sequence of sets $\mathcal{B} = (B_1, \ldots, B_k)$, a \emph{transversal} of $\mathcal{B}$ is a set $\{b_1, \ldots, b_k\}$ such that $b_i \in B_i$ for all $i \in [k]$.
Here we remark that the terms in $\mathcal{B}$ need not be distinct but the elements  in a transversal $\{b_1, \ldots, b_k\}$ are distinct.
For a set of indices $J$, we use $(B_j \mid j\in J)$ to denote the sequence of sets indexed by $J$ so we can use $\mathcal{B} = (B_j \mid j\in [k])$.
A \emph{subsequence} of $\mathcal{B}$ is a sequence $(B_j \mid j\in J)$ for $J\subseteq [k]$.
A \emph{partial transversal} of $\mathcal{B}$ is a transversal of some subsequence of $\mathcal{B}$.
For convenience, we extend all notation regarding transversals and partial transversals to collections of sets.

\begin{proposition}\label{prop:transversalmatroid}
    If $\mathcal{A}$ is a sequence of subsets of a set $S$ and $\mathcal{I}$ is the set of partial transversals of $\mathcal{A}$, then $\mathcal{I}$ is the collection of independent sets of a matroid on $S$.
\end{proposition}

The matroid obtained in \autoref{prop:transversalmatroid} is called a \emph{transversal matroid}.
We use $M[\mathcal{A}]$ to denote the transversal matroid whose independent sets are the partial transversals of $\mathcal{A}$.
We remark that partial transversals correspond to a matching in a bipartite graph with the sets of $\mathcal{A}$ on one part and $S$ on the other, where $A\in \mathcal{A}$ is connected to all of its elements.
The next result follows directly from Ore~\cite{Ore1955}.

\begin{proposition} \label{prop:transversalrank}
    Let $\mathcal{A} = (A_j \mid A\in J)$ be a sequence of subsets of a set $S$.
    If $M[\mathcal{A}]$ has rank function $r$ and $U\subseteq S$, then
    $$ r(U) = \min_{I\subseteq J} \{|J\setminus I| + |\bigcup_{i\in I} (A_i\cap U)|\}.$$
    Furthermore, a maximum size partial transversal of $\mathcal{A}$ inside $U$ can be computed in polynomial time.
\end{proposition}

To simplify notation, by a \emph{family of sets} we mean a sequence of sets in the sense that we allow repetitions ef elements with different indices, but we allow set theoretic notation where such operations should be considered on the indices of the sequence.
To make this observation clear, let $\mathcal{B}_I = (B_i \mid i\in I)$ denote a sequence of sets indexed by $I$, which we call a family of sets $\mathcal{B}_I = \{B_i \mid i\in I\}$
Thus, we use $B_i\in \mathcal{B_I}$ when $i\in I$ and $\mathcal{B}_J \subseteq \mathcal{B}_K$ when $J\subseteq K$, i.e., $\mathcal{B}_J$ is a subsequence of $\mathcal{B}_K$.
We use $|\mathcal{B}_I|$ to denote $|I|$, i,e., the number of elements in the sequence.
Furthermore, when $J\subseteq K$, we use $\mathcal{B}_K \setminus \mathcal{B}_J$ to denote $\mathcal{B}_{K\setminus J}$.

\begin{proposition}\label{prop:gammoidmatroid}
    Let $G$ be a digraph and $X, Y \subseteq V(G)$.
    If $\mathcal{I} = \{\sink(\mathcal{P}) \mid \mathcal{P}$ is a collection of disjoint $X \to Y \text{ paths}\}$, then $\mathcal{I}$ is a collection of independent sets of a matroid on $Y$.
\end{proposition}

The matroid obtained in \autoref{prop:gammoidmatroid} is called a \emph{gammoid} and is denoted by $M(G, X, Y)$.
The rank function of a gammoid can be derived directly from \autoref{thm:Menger}.

\begin{proposition} \label{prop:gammoidrank}
    Let $G$ be a digraph and $X, Y \subseteq V(G)$.
    If $M(G, X, Y)$ has rank function $r$ and $U\subseteq Y$, then $r(U) = \min\{|X| \mid X \text{ is an } (X, U) \text{-separator}\}$.
\end{proposition}

Let $M_1 = (E_1, \mathcal{I}_1), \ldots,M_k = (E_k, \mathcal{I}_k)$ be matroids.
The \emph{union} of these matroids is denoted by $M_1 \vee \cdots \vee M_k = (E_1\cup \cdots \cup E_k, \mathcal{I}_1 \vee \cdots \vee \mathcal{I}_k)$
where $\mathcal{I}_1 \vee \cdots \vee \mathcal{I}_k = \{I_1 \cup \cdots \cup I_k \mid I_1 \in \mathcal{I}_1, \ldots , I_k\in \mathcal{I}_k\}$.
We remark that independent sets in matroids forms a hereditary family so the union of independent sets in $\mathcal{I}_1 \vee \cdots \vee \mathcal{I}_k$ may be considered disjoint.

For the following two results, an \emph{independence oracle} for a matroid $M$ is an algorithm which decides if a subset of $E(M)$ is independent in $M$.
We note that when a matroid $M$ is given as input to an algorithm, it is assumed that the input contains its ground set and an independence oracle for $M$, but $\mathcal{I}(M)$ is not explicitly provided.
Thus, when talking about a polynomial time algorithm for matroids, it means the algorithm runs in polynomial time provided the independence oracle also runs in polynomial time.

\begin{proposition}[Matroid union theorem]
    If $M_1 = (E_1, \mathcal{I}_1), \ldots,M_k = (E_k, \mathcal{I}_k)$ are matroids with rank functions $r_1, \ldots, r_k$, then $M_1 \vee \cdots \vee M_k$ is a matroid with rank function $r$ given by
    $$ r(U) = \min_{T\subseteq U} \{|U\setminus T| + \sum_{i\in [k]}r_i(T\cap E_i)\} $$
    for $U \subseteq E_1 \cup \cdots \cup E_k$.
    Furthermore, for $U\subseteq E_1 \cup \cdots \cup E_k$, $r(U)$ and a set $T$ achieving equality in the above formula can be computed in polynomial time given polynomial time independence oracles for $M_1, \ldots,M_k$.
\end{proposition}

\begin{proposition}[Matroid intersection theorem]
    If $M_1 = (E, \mathcal{I}_1)$ and $M_2 = (E, \mathcal{I}_2)$ are matroids with rank functions $r_1$ and $r_2$, then the maximum size of a set in $\mathcal{I}_1 \cap \mathcal{I}_2$ is equal to
    $$\min_{U\subseteq E} \{r_1(U) + r_2(E\setminus U)\}.$$
    Furthermore, a maximum size common independent set and a $U\subseteq E$ achieving equality in the above formula can be computed in polynomial time given polynomial time independence oracles for $M_1$ and $M_2$.
\end{proposition}

We remark that if we can compute the rank function of a matroid in polynomial time, then we have a polynomial time independence oracle by testing whether the largest independent subset of $U$ has size $|U|$.
Since both gammoids and transversal matroids have polynomial time algorithms for computing the rank of a set, then they also have polynomial time independence oracles.

\subsection{Parameterized complexity}\label{sec:param-compl}
We refer the reader to~\cite{Cygan2015,Downey2013} for background on parameterized complexity, and we define here only the most basic definitions.
A \emph{parameterized problem} is a language $L \subseteq \Sigma^* \times \mathbb{N}$.  For an instance $I=(x,k) \in \Sigma^* \times \mathbb{N}$, $k$ is called the \emph{parameter}.
A parameterized problem $L$ is \emph{fixed-parameter tractable} ({\FPT}) if there exists an algorithm $\mathcal{A}$, a computable function $f$, and a constant $c$ such that given an instance $I=(x,k)$, $\mathcal{A}$   (called an {\FPT} \emph{algorithm}) correctly decides whether $I \in L$ in time bounded by $f(k) \cdot |I|^c$. For instance, the \textsc{Vertex Cover} problem parameterized by the size of the solution is {\FPT}.
A parameterized problem $L$ is in {\sf XP} if there exists an algorithm $\mathcal{A}$ and two computable functions $f$ and $g$ such that given an instance $I=(x,k)$, $\mathcal{A}$  (called an {\sf XP} \emph{algorithm}) correctly decides whether $I \in L$ in time bounded by $f(k) \cdot |I|^{g(k)}$. For instance,  the \textsc{Clique} problem parameterized by the size of the solution is in  {\sf XP}.

Within parameterized problems, the class {\sf W}[1] may be seen as the parameterized equivalent to the class {\sf NP} of classical decision problems. Without entering into details, a parameterized problem being {\sf W}[1]-\emph{hard} can be seen as a strong evidence that this problem is {\sl not} {\FPT}.
The canonical example of {\sf W}[1]-hard problem is \textsc{Clique}  parameterized by the size of the solution.

\subsection{Brambles and directed tree-width}

The \emph{directed tree-width} of digraphs was introduced by Johnson et al.~\cite{Johnson2001} as a directed analogue of tree-width of undirected graphs.
Informally, the directed tree-width $\dtw(G)$ of a digraph $G$ measures how close $G$ can be approximated by a DAG, and the formal definition immediately implies that $\dtw(G) = 0$ if and only if $G$ is an acyclic digraph (DAG).
Directed tree-width and arboreal decompositions are not explicitly used in this article and thus we refer the reader to~\cite{Johnson2001} for the formal definitions.
Here it suffices to mention a few known results.

In the same paper where they introduced directed tree-width, Johnson et al.~\cite{Johnson2001} showed that \textsc{$k$-DDP} can be solved in {\XP} time with parameters $k + \dtw(G)$.
\begin{proposition}[Johnson et al.~\cite{Johnson2001}]
\label{proposition:XP-algo-DDP}
The {\sc $k$-DDP} problem is solvable in time $n^{\Ocal(k+\dtw(G))}$.
\end{proposition}

Notice that every instance of the $(k,c)$-\textsc{DDP} problem with $c \geq k$ is trivially solvable in polynomial time by simply checking for connectivity between each pair $(s_i, t_i)$.
Thus we can always assume that $c < k$.
It is easy to reduce the congested version to the disjoint version of \textsc{DDP}: it suffices to generate a new instance by making $c$ copies of each vertex $v$ of the input digraph,  each of them with the same in- and out-neighborhood as $v$.
A formal proof of this statement was given by Amiri et al.~\cite{AkhoondianAmiri2019}. 
Hence it follows that $(k,c)$-\textsc{DDP} is also {\XP} with parameters $k$ and $\dtw(G)$.
A direct proof of this statement is also possible by applying the same framework used to prove \autoref{proposition:XP-algo-DDP}, although such a proof is not given in~\cite{Johnson2001}.
A similar proof for a congested version of a \textsc{DDP}-like problem is given by Sau and Lopes in~\cite{Lopes2022}.
In any case, the following holds.
\begin{proposition}
\label{proposition:XP-algo-DDP-congestion}
The $(k,c)$-{\sc DDP} problem is solvable in time $(c \cdot n)^{\Ocal(c(k+\dtw(G)))}$.
\end{proposition}

For both the $k$-\textsc{DDP} and $(k,c)$-\textsc{DDP} problems,
(a small variation of) the result of Slivkins~\cite{Slivkins2010} implies that the {\XP} time is unlikely to be improvable to {\FPT}, even when restricted to DAGs (although the result in~\cite{Slivkins2010} concerns the \emph{edge}-disjoint version of $k$-\textsc{DDP}, it easily implies $\W$[1]-hardness of the disjoint version by noticing that the line digraph of a DAG is also a DAG).
To the best of our knowledge, it is not known if an algorithm with running time of the form $f(\dtw(G), k)\cdot n^{g(k)}$; i.e., an {\FPT} algorithm with parameter $\dtw(G)$ for fixed $k$, is possible for $(k,c)$-\textsc{DDP} with $c \geq 1$.

As it is the case with tree-width, Johnson et al.~\cite{Johnson2001} also introduced a dual notion for directed tree-width in the form of \emph{havens}.
However, although the duality in the undirected case is sharp, in the directed case it is only approximate: they showed that the directed tree-width of a digraph $G$ is within a constant factor (more precisely, a factor three) from the maximum order of a haven of $G$.
Since havens and (strict) \emph{brambles} are interchangeable in digraphs whilst paying only a constant factor for the transformation (see~\cite[Chapter 6]{Matthias2014} for example), we skip the definition of the former and focus only on the latter.

\begin{definition}[Brambles in digraphs]
A \emph{bramble} $\mathcal{B} = \{B_1, \ldots, B_\ell\}$ in a digraph $G$ is a collection of strong subgraphs of $G$ such that if $B, B' \in \mathcal{B}$ then $V(B) \cap V(B') \neq \emptyset$ or there are edges in $G$ from $V(B)$ to $V(B')$ and from $V(B')$ to $V(B)$.
We say that the elements of $\mathcal{B}$ are the \emph{bags} of $\mathcal{B}$.
A \emph{hitting set} of a bramble $\mathcal{B}$ is a set $C \subseteq V(G)$ such that $C \cap V(B) \neq \emptyset$ for all $B \in \mathcal{B}$. The \emph{order}  of a bramble $\mathcal{B}$, denoted by $\order(\mathcal{B})$,  is the minimum size of a hitting set of $\mathcal{B}$. The \emph{bramble number} of a digraph $G$, denoted by $\bn(G)$, is the maximum $k$ such that $D$ admits a bramble of order $k$.
A bramble $\mathcal{B}$ is said to be \emph{strict} if for all pairs $B, B' \in \mathcal{B}$ it holds that $V(B) \cap V(B') \neq \emptyset$.
For an integer $c \geq 1$ we say that $\mathcal{B}$ has \emph{congestion} if every vertex of $G$ appears in at most $c$ bags of $\mathcal{B}$.
\end{definition}
See \autoref{fig:bramble-example} for an example of a bramble.
Notice that if $\mathcal{B}$ is a bramble of congestion $c$ for some constant $c$, its \emph{order} increases together with its \emph{size}; i.e., $|\mathcal{B}|$.
More precisely, since every vertex of the host digraph  can hit at most $c$ elements of $\mathcal{B}$ it holds that $\order(\mathcal{B}) \geq \lceil|\mathcal{B}|/c\rceil$.
If $\mathcal{B}'\subseteq \mathcal{B}$ then we may say that $\mathcal{B}'$ is a \emph{subbramble} of $\mathcal{B}$.

\begin{figure}[h]
\centering
\scalebox{.8}{
\begin{tikzpicture}
\node[blackvertex] (c) at (0,0) {};

\node[blackvertex] (u1) at (-1.5,0) {};
\node[blackvertex] (u2) at (1.5,0) {};

\node[blackvertex] (v1) at (-.75,1) {};
\node[blackvertex] (v2) at (.75,1) {};

\node[blackvertex] (w1) at (-0.75,-.75) {};
\node[blackvertex] (w2) at (.75,-.75) {};
\node[blackvertex] (w3) at (0,-1.5) {};

\draw[arrow] (c) -- (u1);
\draw[arrow] (u1) -- (v1);
\draw[arrow] (v1) -- (c);

\draw[arrow] (c) -- (u2);
\draw[arrow] (u2) -- (v2);
\draw[arrow] (v2) -- (c);

\draw[arrow] (w1) -- (w2);
\draw[arrow] (w2) -- (w3);
\draw[arrow] (w3) -- (w1);

\draw[arrow] (w1) -- (c);
\draw[arrow] (c) -- (w2);

\node[draw, ellipse, fit=(c)(u1)(v1), inner sep = -.5pt, yshift = -2.5pt, label =180:{$B_1$}] {};
\node[draw, ellipse, fit=(c)(u2)(v2), inner sep = -.5pt, yshift = -2.5pt, label = 0:{$B_2$}] {};
\node[draw, ellipse, fit=(w1)(w2)(w3), inner sep = 0pt, label =0:{$B_3$}, yshift=2pt] {};

\end{tikzpicture}}%
\caption{Example of a bramble $\{B_1, B_2, B_3\}$ of order two.}
\label{fig:bramble-example}
\end{figure}
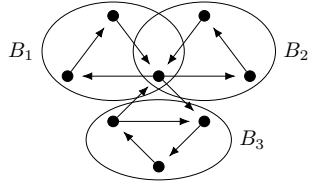

Johnson et al.~\cite{Johnson2001} gave an algorithm that, given a digraph $G$, either correctly decides that $\dtw(G) \leq 3k-2$ (also yielding an arboreal decomposition of $G$) or produces a bramble of order $\lfloor k/2 \rfloor$.
This was later improved to an {\FPT} algorithm by Campos et al.~\cite{Campos2022}.
Although the authors do not explicitly say that the produced bramble is strict, it is easy to verify that this is the case in their proof, and the same holds for the proof of~\cite{Johnson2001}.
\begin{proposition}[Campos et al.~\cite{Campos2022}]
\label{prop:FPT-dtw-or-bramble}
Let $G$ be a digraph and $t$ be a non-negative integer.
There is an algorithm running in time $2^{\Ocal(t \log t)}\cdot n^{\Ocal(1)}$ that either produces an arboreal decomposition of $G$ of width at most $3t-2$ or finds a strict bramble of order $\lfloor t/2 \rfloor$ in $G$.
\end{proposition}

Brambles of constant congestion are a key structure used to solve instances of $(k,c)$-\textsc{DDP} in $f(k)$-strong digraphs in the approach by Edwards et al.~\cite{Edwards2017}.
Briefly, given an instance $(G, S, T)$ of this problem and a large bramble $\mathcal{B}$ of congestion two in $G$, they find two sets of paths $\mathscr{P}^S$ from $S$ to the bags of the bramble and $\mathscr{P}^T$ from the bags of the bramble to $T$ with a set of properties that can be exploited to appropriately connect the last vertices of paths in the first set with the first vertices of the paths from the second set, thus constructing a solution to $(G, S, T)$.

If the directed tree-width of $G$ is bounded, then one can solve $(G, S, T)$ by applying \autoref{proposition:XP-algo-DDP-congestion}.
Edwards et al.~\cite{Edwards2017} showed that every digraph of sufficiently large directed tree-width is guaranteed to contain a large bramble of congestion at most two and that, given appropriate structures that are guaranteed to be found in such digraphs, one can construct a bramble with the desired properties in polynomial time.
Formally, they use the following result, originally proved by Kawarabayashi and Kreutzer~\cite{KawarabayashiK15} where an {\XP} algorithm is given, and then improved by Campos et al.~\cite{Campos2022} with an {\FPT} algorithm and a better dependency on $k$.
We refer the reader to~\cite{KawarabayashiK15,Campos2022} for the definition of well-linked sets, and we remark that, for convenience, we present the statement of the following result in a slightly different way than in the original article.
\begin{proposition}[Campos et al~\cite{Campos2022}]\label{prop:well-linked-set-and-path}
Let $g(k) =(t+1)(\lfloor t/2\rfloor +1) - 1$ and $G$ be a digraph with directed tree-width at least $3g(t)-1$.
There is an algorithm running in time $2^{\Ocal(t^2 \log t)}\cdot n^{\Ocal(1)}$ that finds in $G$ a bramble $\mathcal{B}$ of order $g(t)$, a path $P$ that intersects every bag of $\mathcal{B}$, and a well-linked set $X$ of size $t$ such that $X \subseteq V(P)$.
\end{proposition}
\begin{proposition}[Edwards et al.~\cite{Edwards2017}]\label{prop:finding-bramble-congestion-two}
There exists a function $f: \mathbb{N} \to \mathbb{N}$ satisfying the following.
Let $G$ be a digraph and $t \geq 1$ be an integer.
Let $P$ be a path in $G$ and $X \subseteq V(P)$ be a well-linked set with $|X| \geq f(t)$.
Then $G$ contains a bramble $\mathcal{B}$ of congestion two and size $t$ and, given $G, P$, and $X$, we can find $\mathcal{B}$ in polynomial time.
\end{proposition}

Pipelining \cref{prop:FPT-dtw-or-bramble,prop:well-linked-set-and-path,prop:finding-bramble-congestion-two} we obtain the following. 
\begin{corollary}\label{cor:bounded-dtw-or-bramble-congestion-two}
There is a function $f: \mathbb{N} \to \mathbb{N}$ and an {\FPT} algorithm with parameter $t$ that, given a digraph $G$ and an integer $t \geq 1$, either correctly decides that the directed tree-width of $G$ is at most $f(t)$ or finds a bramble $\mathcal{B}$ of congestion two and size $t$ in $G$.
\end{corollary}

\section{New Menger-like statements for paths in digraphs}
\label{sec:newpaths}

In this section we present the definition and the min-max formulas associated with each pair \Dpaths/\Dcut{s}, \Tpaths/\Tcut{s}, and \Rpaths/\Rcut{s}.
For all three pairs, we prove in this section that there is a sharp duality between the maximum number of paths and the minimum order of the associated cut.
Since all three types of paths share some properties (in fact, the major distinction between them is in how they reach their destinations), it is convenient to adopt the following notations.

\begin{definition}[Digraph-source sequences and respecting paths]\label{def:graph-source-sequence}
For an integer $\ell \geq 1$, a \emph{digraph-source sequence} of \emph{size} $\ell$ is a pair $(\mathcal{F}, \mathcal{S})$ such that $\mathcal{F} = (G_1, \ldots, G_\ell)$ is a sequence of digraphs and $\mathcal{S} = (S_1, \ldots, S_\ell)$ is a sequence of subsets of vertices with $S_i \subseteq V(G_i)$ for $i\in [\ell]$. We say that a set of paths $\mathscr{P}$ \emph{respects} $(\mathcal{F, S})$ or, alternatively, is \emph{$(\mathcal{F, S})$-respecting} if there is a partition $\mathcal{P}_1, \ldots, \mathcal{P}_\ell$ of $\mathscr{P}$ such that
\begin{enumerate}[(a)]
\item for $i \in [\ell]$, $\mathcal{P}_i$ is a set of disjoints paths in $G_i$, and
\item for $i \in [\ell]$, $\source(\mathcal{P}_i)\in S_i$.
\end{enumerate}
In this case, we say that $\mathcal{P}_1, \ldots, \mathcal{P}_\ell$ is a \emph{defining partition} of $\mathscr{P}$.
\end{definition}
Thus in any set of $(\mathcal{F, S})$-respecting paths, any two paths can intersect only if they belong to distinct parts of the defining partition.

To provide some intuition within the context of the $(k,c)$-\textsc{DDP} problem, in the next three definitions one can think of the sequence $(S_1, \ldots, S_\ell)$ as being formed by many copies of $S$ followed by many copies of $T$.
Notice that any set of $(\mathcal{F, S})$-respecting paths includes paths \emph{leaving} $T$, which seems counter-intuitive when considering the goal of solving instances of $(k,c)$-\textsc{DDP}.
We remark this can be easily addressed by associating, with each copy of $T$ in the sequence $\mathcal{S}$, the digraph $G^{\sf rev}$ formed by reversing the orientation of every edge of $G$ (we use this tool in \autoref{section:DDP-algorithm} and a formal definition of $G^{\sf rev}$ is given in \autoref{def:reverse-digraph}).

\medskip
\noindent\textbf{\Dpaths and \Dcut{s}.}
Within the context of this paper, \Dpaths and \Dcut{s} are used as a tool to prove our results regarding \Tpaths and \Tcut{s}. 
We believe that \Dpaths/\Dcut{s} are interesting on their own and under a pedagogical point of view, since they are easy to visualize and comprehend.
Thus the definition and proofs regarding those objects are included and we expect that they find further applications in the context of problems like the $(k,c)$-\textsc{DDP} problem.
In this scenario, one should think of the set $B$ in the following definition as the set of vertices of a highly connected structure that is intended to be used to appropriately connect the last vertices of paths from $S$ to the first vertices of paths from $T$.

\begin{definition}[\Dpaths and \Dcut{s}]\label{def:D-paths-and-D-cuts}
For an integer $\ell \geq 1$, let $(\mathcal{F, S})$ be a digraph-source sequence with $\mathcal{F} = (G_1, \ldots, G_\ell)$, $\mathcal{S} = (S_1, \ldots, S_\ell)$, and $B \subseteq \bigcup_{i = 1}^{\ell}V(G_i)$.
With respect to $B$, we say that a set of $(\mathcal{F, S})$-respecting paths $\mathscr{P}$ with defining partition $\mathcal{P}_1, \ldots, \mathcal{P}_\ell$ is a set of \emph{\Dpaths} if 
\begin{bracketenumerate}
\item for all distinct $P,P' \in \mathscr{P}$ it holds that $\sink(P) \neq \sink(P')$, and
\item[\lipItem{(2a)}] $\sink(\mathscr{P}) \subseteq B$.
\end{bracketenumerate}
A \emph{\Dcut} is a sequence $\mathcal{X} = (X_0, \ldots, X_\ell)$ with $X_0 \subseteq B$ such that, for $i \in [\ell]$, the set $X_i \subseteq V(G_i)$ is an $(S_i, B\setminus X_0)$-separator in $G_i$.
The \emph{order} of a \Dcut $X$ is $\order(\mathcal{X}) = |X_0| + \sum_{i=1}^{\ell}|X_i|$.
\end{definition}

Thus, in the definition of \Dpaths we ask each collection of paths associated with each $\mathcal{P}_i$ to be pairwise disjoint in $G_i$, and the paths from distinct parts $\mathcal{P}_i, \mathcal{P}_j$ may share vertices in $\bigcup_{i=1}^{\ell}V(G_i)$ other than the last vertices of the paths.
The ``{\sf D}'' in the name stands for ``disjoint''.
For the min-max formula, we prove the following.
\begin{theorem}\label{theo:min-max-statement-var-1}
Let $(\mathcal{F, S})$ be digraph-source sequence of size $\ell$ with $\mathcal{F} = (G_1, \ldots, G_\ell)$, and let $B \subseteq  \bigcup_{i=1}^{\ell}V(G_i)$.
With respect to $\mathcal{F, S}$, and $B$, the maximum number of \Dpaths is equal to the minimum order of a \Dcut.
Additionally, a \Dcut of minimum order and a maximum collection of \Dpaths can be found in time $\Ocal((\ell\cdot n^* + |\mathcal{B}|)^2)$ where $n^* = \max_{i \in [\ell]}(|V(G_i)|)$.
\end{theorem}

\medskip
\noindent\textbf{\Tpaths and \Tcut{s}.}
For the sake of intuition, in the next definition one should think of $\mathcal{B}$ as a bramble.
Informally, and given a digraph $G$, we use \Tpaths and \Tcut{s} to find a large collection of paths from a given ordered $S \subseteq V(G)$ to the bags of a subbramble $\mathcal{B}^S \subseteq \mathcal{B}$, and from the bags of another subbramble $\mathcal{B}^T \subseteq \mathcal{B}$ to $T$ (we can achieve this orientation for these paths by reversing the orientation of the edges of $G$), while ensuring that every vertex outside of $\mathcal{B}$ appears in at most two of those paths, and that all elements of $\mathcal{B}^S \cup \mathcal{B}^T$ are pairwise distinct.
The first property can be achieved by simply applying Menger's Theorem (cf. \autoref{thm:Menger}) twice.
However, by doing this, we can end with a set of paths all ending on the same bag of $\mathcal{B}$.
This scenario is far from ideal, since at some point the goal is to use the strong connectivity of $G[B \cup B']$ for every $B,B' \in \mathcal{B}$ to appropriately connect the ending vertices of the paths from $S$ to the starting vertices of the paths to $T$, while maintaining the property that every vertex appears in at most two (or $c$, in the general case) of those paths.
If a unique bag $B$ contains all starting and ending vertices of the paths, then connecting those vertices while maintaining such properties may be as hard as finding a solution to an instance of $(k,c)$-\textsc{DDP} in the strong digraph $G[B]$, or downright impossible to do.

Therefore, in the proofs applying similar techniques, as seen in the works by Edwards et al.~\cite{Edwards2017} and by Giannopoulou et al.~\cite{Giannopoulou2020}, there is considerable effort into finding paths with ``good properties'' that can be used to connected the paths inside of $\mathcal{B}$ (or inside of a cylindrical grid in the case of~\cite{Giannopoulou2020}), and these properties always include, as far as we know, that the paths end or start in distinct bags of $\mathcal{B}$ (or distinct sections of the cylindrical grid).
In particular,~\cite[Lemma 16]{Edwards2017} includes a set of seven properties over a set of paths that we can substitute by \Tpaths to achieve better results in a simpler manner.

We can prove the same results using \Rpaths and \Rcut{s}, which are simpler than \Tpaths and \Tcut{s}.
We include the proofs for the two latter objects for potential applications in which the extra properties of \Tpaths and \Tcut{s} may become handy.

\begin{definition}[\Tpaths and \Tcut{s}]\label{def:t-paths-and-t-cuts}
For an integer $\ell \geq 1$, let $(\mathcal{F, S})$ be a digraph-source sequence with $\mathcal{F} = (G_1, \ldots, G_\ell)$ and $\mathcal{S} = (S_1, \ldots, S_\ell)$, and $\mathcal{B}$ be a family of subsets of $\bigcup_{i = 1}^{\ell}V(G_i)$.
With respect to $\mathcal{B}$, we say that a set of $(\mathcal{F, S})$-respecting paths $\mathscr{P}$ with defining partition $\mathcal{P}_1, \ldots, \mathcal{P}_\ell$ is a set of \emph{\Tpaths} if
\begin{bracketenumerate}
\item for all distinct $P,P' \in \mathscr{P}$ it holds that $\sink(P) \neq \sink(P')$, and
\item[\lipItem{(2b)}] the set $\sink(\mathscr{P})$ is a partial transversal of $\mathcal{B}$.
\end{bracketenumerate}
A \emph{\Tcut} is a pair $(\mathcal{B'},\mathcal{X})$ with $\mathcal{B}' \subseteq \mathcal{B}$ and such that $\mathcal{X}$ is a \Dcut with respect to $\mathcal{F, S}$, and $\cupall\mathcal{B}'$.
The \emph{order} of a \Tcut $(\mathcal{B'}, X)$ is $\order(\mathcal{B'}, X) = |\mathcal{B} \setminus \mathcal{B}'| + \order(\mathcal{X})$.
\end{definition}
For convenience, we keep only one set of parenthesis, writing $\order(\mathcal{B'}, \mathcal{X})$ instead of $\order((\mathcal{B}', \mathcal{X}))$.


Notice that conditions \lipItem{(1)} in the definition of \Tpaths is the same as in the definition of \Dpaths.
Thus the difference between \Dpaths and \Tpaths is that in the former we ask the paths to end in distinct vertices of $B$, while in the latter we ask the endpoints of the paths to form a partial transversal of $\mathcal{B}$.
This implies that those endpoints are distinct, and that each of them is associated with a unique element of $\mathcal{B}$.
The ``{\sf T}'' in the name stands for ``transversal''.
See \autoref{fig:transversal-example} for an example of a transversal of a collection of sets.
\begin{figure}[h]
\centering
\begin{tikzpicture}
\node[draw, ellipse, inner ysep = 1.25cm, inner xsep=.5cm, label=90:{$B_4$}] (a) at (0,0) {};
\foreach \i/\j in {-1.2/1,-0.3/2,.55/3}{
  \node[draw, ellipse, inner ysep = .25cm, inner xsep=1cm, label=0:{$B_\j$}] (e\i) at (0.3,\i) {};
  \node[blackvertex, label = 180:{$v_\j$}] (v\j) at (0.1,\i) {};
}
\node[blackvertex, label=180:{$v_4$}] at ($(v3) + (0,.7)$) {};

\end{tikzpicture}
\caption{Example of a transversal of the collection $\{B_1, B_2, B_3, B_4\}$. For $i \in [4]$ the vertex $v_i$ is associated with the set $B_i$.}
\label{fig:transversal-example}
\end{figure}
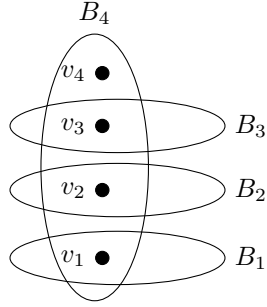

For the \Tpaths/\Tcut{s} duality, we prove the following.
\begin{theorem}\label{theo:min-max-statement-var-2}
Let $(\mathcal{F, S})$ be a digraph-source sequence of size $\ell$ with $\mathcal{F} = (G_1, \ldots, G_\ell)$, and let $\mathcal{B}$ be a collection of subsets of $\bigcup_{i=1}^{\ell}V(G_i)$.
With respect to $\mathcal{F, S}$, and $\mathcal{B}$, the maximum number of \Tpaths is equal to the minimum order of a \Tcut.
Additionally, a \Tcut of minimum order and a maximum collection of \Tpaths can be found in time $\Ocal((\ell\cdot n^* + |\mathcal{B}|)^2)$, where $n^* = \max_{i \in [\ell]}(|V(G_i)|)$.
\end{theorem}

\medskip
\noindent\textbf{\Rpaths and \Rcut{s}.}
The intuition for \Rpaths is similar to the one for \Tpaths, as is the motivation to use these objects in the context of $(k,c)$-\textsc{DDP} and similar problems.
The difference between them is that, if $\mathscr{P}$ is a set of \Tpaths, then all vertices of the form $\sink(P)$, where $P \in \mathscr{P}$, are distinct.
In \Rpaths this only holds when considering paths inside of the same part of its defining partition.
More precisely, given a partition $\mathcal{P}$ of a set of \Rpaths as defined below, $\sink(P)$ and $\sink(P')$ are guaranteed to be disjoint only when $P$ and $P'$ are in distinct parts of $\mathcal{P}$.
In \autoref{section:DDP-algorithm} we show that this relaxation poses no problem for the application of \Rpaths/\Rcut{s} we show in this article.

\begin{definition}[\Rpaths and \Rcut{s}]\label{def:r-paths-and-r-cuts}
For $\ell \geq 1$, let $(\mathcal{F, S})$ be a digraph-source sequence with $\mathcal{F} = (G_1, \ldots, G_\ell)$ and $\mathcal{S} = (S_1, \ldots, S_\ell)$, and $\mathcal{B}$ be a family of subsets of $\bigcup_{i = 1}^{\ell}V(G_i)$.
With respect to $\mathcal{B}$, we say that a set of $(\mathcal{F, S})$-respecting paths $\mathscr{P}$ with defining partition $\mathcal{P}_1, \ldots, \mathcal{P}_\ell$ is a set of \emph{\Rpaths} if
\begin{bracketenumerate}
\item[\lipItem{(1c)}] for some family $\mathcal{B}^* \subseteq \mathcal{B}$ there is a bijective mapping $h: \mathscr{P} \to \mathcal{B}^*$ such that $h(P) = B$ implies $\sink(P) \in B$.
\end{bracketenumerate}
An \emph{\Rcut} is a pair $(\mathcal{B',X})$ where $\mathcal{B}' \subseteq \mathcal{B}$ and $\mathcal{X}$ is a sequence $(X_1, \ldots, X_\ell)$ such that each $X_i \in \mathcal{X}$ is an $(S_i, \cupall\mathcal{B}')$-separator in $G_i$.
The \emph{order} of an \Rcut $(\mathcal{B', X})$ is $\order(\mathcal{B', X}) = |\mathcal{B} \setminus \mathcal{B}'| + \sum_{i=1}^{\ell}|X_i|$.
\end{definition}
We remark that the only difference between \Rpaths and \Tpaths is that, in the latter, condition \lipItem{(1c)} ensures that all vertices forming the partial transversal of $\mathcal{B}$ are distinct.
For \Rpaths this is not the case, and the construction of the auxiliary digraph used in the proof of the \Rpaths/\Rcut{s} duality reflects this distinction.
The ``{\sf R}'' in the name stands for ``representatives''.
For the duality, we prove the following.
\begin{theorem}\label{theo:min-max-statement-var-3}
Given a digraph-source sequence $(\mathcal{F, S})$ of size $\ell$ and a set $B \subseteq  \bigcup_{i=1}^{\ell}V(G_i)$ then, with respect to $\mathcal{F, S}$, and $B$, the maximum number of \Rpaths is equal to the minimum order of an \Rcut.
Additionally, an \Rcut of minimum order and a maximum collection of \Rpaths can be found in $\Ocal((k\cdot n^* + |\mathcal{B}|)^2)$ where $n^* = \max_{i \in [\ell]}(|V(G_i)|)$.
\end{theorem}

Observe that the right side of the pair forming an \Rcut cannot be simply a set of vertices $X$ because, for example, a vertex $v \in X$ can be in two distinct $G_i$ and $G_j$ and be part of the separator in $G_i$ but {\sl not} part of the separator in $G_j$.
In this case $v$ would be counted twice in the order of the \Rcut, but it is only used in one separator.
Also notice that when $|\mathcal{B}|$ is larger than the allowed budget to construct an \Rcut, every \Rcut of appropriate order must identify some separator in some $G_i$, i.e., $\mathcal{X} \neq \emptyset$.
In fact, there are only two options for the size of $\mathcal{X}$: either $\mathcal{B}' = \emptyset$ and hence $\mathcal{X} = \emptyset$, or $\mathcal{B}' \neq \emptyset$ and $|\mathcal{X}| = \ell$.
In the latter case, it is possible that some $X_i \in \mathcal{X}$ are empty.
In \autoref{lem:R-paths-or-small-cut} we exploit this fact to show how to either find in a digraph $G$ a large collection of \Rpaths from a set $S$ to the vertices appearing in the elements of some sufficiently large collection $\mathcal{B}$ (corresponding to a bramble), or a small separator intersecting every path from $S$ to all such vertices.

Next, we present the proofs of the min-max relations (\cref{theo:min-max-statement-var-1,theo:min-max-statement-var-2,theo:min-max-statement-var-3}) in two different ways.
In \autoref{subsection:proofs-by-matroids} we show how to use classical results of matroids to obtain the relations.
In \autoref{subsection:proofs-by-menger} we show how to prove prove the results by applying Menger's Theorem to carefully constructed auxiliary digraphs.
We remark that the bounds on the running times of all three theorems are obtained from the second approach.

\subsection{Alternative proofs using matroids}

\label{subsection:proofs-by-matroids}

In this section we show how to prove \cref{theo:min-max-statement-var-1,theo:min-max-statement-var-2,theo:min-max-statement-var-3} by using matroid union and intersection from gammoids and transversal matroids.
For the remainder of this subsection, let $(\mathcal{F, S})$ be a digraph-source sequence with $\mathcal{F} = (G_1, \ldots, G_\ell)$ and $\mathcal{S} = (S_1, \ldots, S_\ell)$, for a positive integer $\ell$.

We remark that each proof will relate \Dpaths, \Tpaths and \Rpaths to some application of matroid union or intersection starting from gammoids or transversal matroids.
In what follows, we only prove the min-max relation between a maximum collection of paths and the minimum order of a cut.
We ignore the algorithmic statement of those results for two reasons.
First, the provided proofs can be turned into polynomial time algorithms simply by using the polynomial time algorithms for matroid union and intersection.
And second because the complexity of the matroid algorithms is worse than the complexity of the corresponding algorithms obtained in \autoref{subsection:proofs-by-menger}.

\bigskip
\noindent\textbf{Proof for \Dpaths and \Dcut{s}.}
%
%
Let $B \subseteq  \bigcup_{i=1}^{\ell}V(G_i)$ and $M_i$ be the gammoid $M(G_i, S_i, B\cap V(G_i))$.
Let $\mathscr{P}$ be a $(\mathcal{F, S})$-respecting collection of \Dpaths with respect to $B$ and let $\mathcal{P}_1, \ldots, \mathcal{P}_\ell$ be a defining partition of $\mathscr{P}$.
The key observation here is that $\sink(\mathcal{P}_i)$ is independent in $M_i$ for $u\in [\ell]$, which implies $\sink(\mathscr{P})$ is independent in $M_1 \vee \cdots \vee M_\ell$.
Since the independent sets in $M_1 \vee \cdots \vee M_\ell$ are disjoint unions of independent sets, we conclude the independent sets in $M_1 \vee \cdots \vee M_\ell$ are precisely the sets $\sink(\mathscr{P})$ where $\mathscr{P}$ are \Dpaths with respect to $\mathcal{F}$, $\mathcal{S}$ and $B$.
We immediately conclude the following result.

\begin{proposition} \label{prop:dpathmatroid}
    Let $B \subseteq  \bigcup_{i=1}^{\ell}V(G_i)$ and $M_i$ be the gammoid $M(G_i, S_i, B\cap V(G_i))$ with rank function $r_i$, for $i\in [\ell]$.
    If $\mathcal{I} = \{\sink(\mathscr{P}) \mid \mathscr{P}\text{ is a collection of \Dpaths with respect to } B\}$, then $\mathcal{I}$ is the collection of independent sets of a matroid on $B$ with rank function $r$ given by
    $$ r(U) = \min_{T \subseteq U} \{|U\setminus T| + \sum_{i\in [\ell]}r_i(T\cap V(G_i))\}$$
    for $U\subseteq B$.
\end{proposition}

%

The matroid obtained in \autoref{prop:dpathmatroid} is called a \emph{\Dpaths matroid} and we denote it by $M(\mathcal{F}, \mathcal{S}, B)$.

Now, we can prove \autoref{theo:min-max-statement-var-1} by showing that the formula provided for $r(B)$ minimizes the order of a \Dcut.
Let $\mathcal{X} = (X_0, \ldots, X_\ell)$ be a \Dcut.
For $T = B\setminus X_0$, we get that $X_i$ is a $(S_i, T \cap V(G_i))$-separator which implies $|X_i| \ge r_i(T\cap V(G_i))$, for $i\in [\ell]$, and we get $\order(\mathcal{X}) \ge r(B)$.
To prove equality, let $T\subseteq B$ achieve equality for the formula of $r(B)$.
Let $X_0 = B\setminus T$ and choose the \Dcut $\mathcal{X} = (X_0, \ldots, X_\ell)$ by letting $X_i$ be a $(S_i, B\setminus X_0)$-separator of size $r_i(T\cap V(G_i))$.
This can be done since $(B\setminus X_0)\cap V(G_i) = T\cap V(G_i)$.


\bigskip
\noindent\textbf{Proof for \Tpaths and \Tcut{s}.} Let $\mathcal{B}$ be a family of subsets of $\bigcup_{i=1}^{\ell}V(G_i)$.
We note that, by definition, a collection of $(\mathcal{F, S})$-respecting paths $\mathscr{P}$ are \Tpaths with respect to $\mathcal{B}$ if, and only if, they are \Dpaths with respect to $\cupall \mathcal{B}$ and $\sink(\mathscr{P})$ is a partial transversal of $\mathcal{B}$.
Let $M_1$ be the \Dpaths matroid $M(\mathcal{F, S}, \cupall \mathcal{B})$ with rank function $r_1$ and $M_2$ be the transversal matroid $M[\mathcal{B}]$ with rank function $r_2$.
We conclude that $\{\sink(\mathscr{P}) \mid \mathscr{P}\text{ is a collection of \Tpaths with respect to }\mathcal{F}, \mathcal{S}\text{ and }\mathcal{B}\} = \mathcal{I}(M_1) \cap \mathcal{I}(M_2)$.

By using the matroid intersection theorem, we get that the maximum size of $\sink(\mathscr{P})$, and equivalently of $\mathscr{P}$, is equal to $$\min_{U\subseteq \cupall \mathcal{B}} \{r_1(U) + r_2(\cupall \mathcal{B}\setminus U)\}.$$
From \autoref{prop:dpathmatroid} and \autoref{prop:transversalrank} we simplify the previous equation into
$$ \min_{U\subseteq \cupall \mathcal{B}, T\subseteq U, \mathcal{B}' \subseteq \mathcal{B}}\left\{ |\mathcal{B}\setminus \mathcal{B}'| + |\cupall \mathcal{B}' \setminus U| + |U\setminus T| + \sum_{i\in [\ell]} |X_i|\right\}$$
where $X_i$ is a minimum size $(S_i, T\cap V(i))$-separator.
Considering only the value of $U$, note that this equation is minimized when $U = \cupall \mathcal{B}'$.
By renaming $\cupall \mathcal{B}'\setminus T$ by $X_0$ we get that the maximum size of $\mathscr{P}$ is
$$ \min_{\mathcal{B}' \subseteq \mathcal{B}, X_0\subseteq \cupall \mathcal{B}'}\left\{ |\mathcal{B}\setminus \mathcal{B}'| + |X_0| + \sum_{i\in [\ell]} |X_i|\right\}$$
where $X_i$ is a minimum size $(S_i, \cupall \mathcal{B}' \setminus X_0)$-separator.
We note the last equation is equal to $\order(\mathcal{B}', \mathcal{X})$ for some \Tcut with respect to $\mathcal{F}$, $\mathcal{S}$, and $\mathcal{B}$.
To complete the proof, note that for any \Tcut $(\mathcal{B}', \mathcal{X})$, we directly get
$$\order(\mathcal{B}', \mathcal{X}) \ge \min_{\mathcal{B}' \subseteq \mathcal{B}, X_0\subseteq \cupall \mathcal{B}'}\left\{ |\mathcal{B}\setminus \mathcal{B}'| + |X_0| + \sum_{i\in [\ell]} |X_i|\right\}$$
where $X_i$ is a minimum size $(S_i, \cupall \mathcal{B}' \setminus X_0)$-separator.

\bigskip
\noindent\textbf{Proof for \Rpaths and \Rcut{s}.}
Let $\mathcal{B}$ be a family of subsets of $\bigcup_{i=1}^{\ell}V(G_i)$.
We remark that this proof is very similar to the proof obtained for the min-max duality of \Tpaths and \Tcut{s}.
The difference is that we construct the \Dpaths matroid on disjoint copies of the graphs in $\mathcal{F} = (G_1, \ldots, G_\ell)$ and moving the sets in $\mathcal{S} = (S_1, \ldots, S_\ell)$ to the copied vertices.
In this construction, the sets in $\mathcal{B}$ maintain all copies of vertices it contained initially.
In this case, a partial transversal of the modified problem is a bijection between paths and a subfamily of $\mathcal{B}$, and the paths between different parts of a defining partition have no other constraint.
We remark that creating copies of the graphs in $\mathcal{F}$ is equivalent to assuming that $G_1, \ldots, G_\ell$ are disjoint and following the same proof for \Tpaths and \Tcut{s}.


\subsection{Menger-based proofs}\label{subsection:proofs-by-menger}

In this section we show how to obtain \cref{theo:min-max-statement-var-1,theo:min-max-statement-var-2,theo:min-max-statement-var-3} applying Menger's Theorem (cf. \autoref{thm:Menger}) to auxiliary digraphs constructed from given the digraph-source sequences.
Since all constructions begin from the same digraph, it is convenient to adopt the following notation.
\begin{definition}
\label{def:disjoint-copies-digraph}
Let $\mathcal{F} = (G_1, \ldots, G_\ell)$ be a sequence of (not necessarily disjoint) digraphs.
We denote by \emph{$G(\mathcal{F})$} the digraph formed by taking disjoint copies of each $G_i$.
Formally, for each $i \in [\ell]$ and $v \in V(G_i)$ we add to $G(\mathcal{F})$ the vertex $v_i$, set $V_i = \{v_i \mid v \in V(G_i)\}$, and say that $v_i$ is a \emph{copy} of $v$.
Then, for each $u_i,v_i \in V_i$ we add the edge $(u_i, v_i)$ to $G(\mathcal{F})$ if an only if $(u,v) \in G_i$.
\end{definition}
In other words, in $G(\mathcal{F})$ each subgraph induced by $V_i$, for $i \in [\ell]$, is exactly a copy of $G_i$, and all such copies are disjoint.
For the remaining of this subsection, we assume that $(\mathcal{F, S})$ is a digraph-source sequence with $\mathcal{F} = (G_1, \ldots, G_\ell)$ and $\mathcal{S} = (S_1, \ldots, S_\ell)$.

\medskip
\noindent\textbf{Proof for \Dpaths and \Dcut{s}.}
In this part, we assume that $B \subseteq \bigcup_{i=1}^{\ell}V(G_i)$.
From $\mathcal{F, S}$, and $B$ we construct an auxiliary digraph $G^{\sf D}$ containing $B$ and a set of vertices $S'$ such that, from any set of disjoint $S' \to B$ paths in $G^{\sf D}$, we can construct an equally large set of \Dpaths.

We begin with $G^{\sf D} = G(\mathcal{F})$.
Then, we add to $G^{\sf D}$ all vertices of $B$ and an edge from each $u \in \bigcup_{i=1}^{\ell}V^i$ to $v \in B$ if $u$ is a copy of $v$.
Define $S'_i = \{s_i \mid s \in S_i\}$ and $S' = \bigcup_{i = 1}^{\ell}S'_i$.
Notice that $B$ is an independent set of $G^{\sf D}$ and that each $v \in B$ has at least one copy in some $V_i$.
See \autoref{fig:construction-D-paths} for an illustration of this construction.

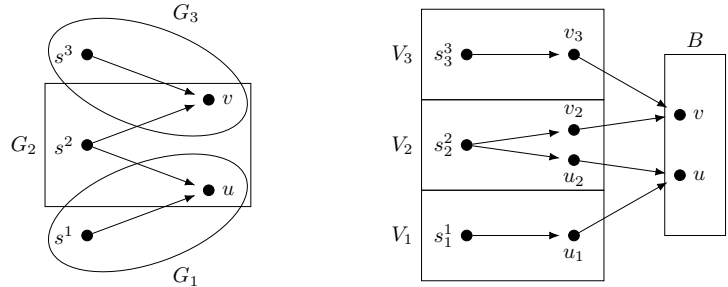
\begin{figure}[h]
\centering
\scalebox{.8}{
\begin{tikzpicture}
\node[draw, rectangle, minimum width=3cm, minimum height=1.5cm, label=180:{$V_1$}] (d1) at (0,0) {};
\node[blackvertex, label=180:{$s^1_1$}] (s1) at ($(d1.west) + (.75, 0)$) {};
\node[draw, rectangle, minimum width=3cm, minimum height=1.5cm, label=180:{$V_2$}] (d2) at (0,1.5) {};
\node[blackvertex, label=180:{$s^2_2$}] (s2) at ($(d2.west) + (.75, 0)$) {};
\node[draw, rectangle, minimum width=3cm, minimum height=1.5cm, label=180:{$V_3$}] (d3) at (0,3) {};
\node[blackvertex, label=180:{$s^3_3$}] (s3) at ($(d3.west) + (.75, 0)$) {};
\node[draw, rectangle, label=90:{$B$}, minimum width=1cm, minimum height = 3cm] (B) at (3,1.5) {};

\node[blackvertex, label=-90:{$u_1$}] (u1) at ($(d1.east) + (-.5, 0)$) {};

\node[blackvertex, label=-90:{$u_2$}] (u2) at ($(d2.east) + (-.5, -0.25)$) {};
\node[blackvertex, label=90:{$v_2$}] (v2) at ($(d2.east) + (-.5, 0.25)$) {};

\node[blackvertex, label=90:{$v_3$}] (v3) at ($(d3.east) + (-.5, 0)$) {};

\node[blackvertex, label=0:{$u$}] (u) at ($(B) + (-0.25, -.5)$) {};
\node[blackvertex, label=0:{$v$}] (v) at ($(B) + (-0.25, .5)$) {};

\draw[arrow] (s1) -- (u1);

\draw[arrow] (s2) -- (u2);
\draw[arrow] (s2) -- (v2);

\draw[arrow] (s3) -- (v3);

\draw[arrow] (v3) -- (v);

\draw[arrow] (v2) -- (v);
\draw[arrow] (u2) -- (u);

\draw[arrow] (u1) -- (u);

\begin{scope}[xshift=-7cm]
\node[blackvertex, label=180:{$s^1$}] (s1) at (0,0) {};
\node[blackvertex, label=180:{$s^2$}] (s2) at (0,1.5) {};
\node[blackvertex, label=180:{$s^3$}] (s3) at (0,3) {};

\node[blackvertex, label=0:{$u$}] (u) at ($(s2) + (2,-.75)$) {};
\node[blackvertex, label=0:{$v$}] (v) at ($(s2) + (2,.75)$) {};
\node[draw, ellipse, fit = (s3)(v), rotate = -20, label=90:{$G_3$}] {};
\node[draw, ellipse, fit = (s1)(u), rotate = 20, label=-90:{$G_1$}] {};
\node[draw, rectangle, fit = (s2)(v)(u), inner xsep = 17pt, inner ysep=5pt, label=180:{$G_2$}] {};
\draw[arrow] (s1) -- (u);
\draw[arrow] (s2) -- (u);
\draw[arrow] (s2) -- (v);
\draw[arrow] (s3) -- (v);

\end{scope}
\end{tikzpicture}%
}
\caption{On the left, digraphs $G_1$, $G_2$, and $G_3$, with $S_i = \{s^i\}$ for $i \in [3]$.
In the example $B = \{u,v\}$ forms the set of destinations for the \Dpaths we want to construct, taken from the vertex set of the digraphs $G_i$.
On the right, the resulting construction of $G^{\sf D}$. Notice that $v_2$ and $v_3$ are copies of $v$, that $u_1$ and $u_2$ are copies of $u$, and that each $s^i_i$ is a copy of $s_i$.}
\label{fig:construction-D-paths}
\end{figure}

Now, for $i \in [\ell]$ and $B' \subseteq B$, let $x_i(B')$ be the size of a minimum $(S_i, B')$-separator in $G_i$ and
\begin{equation}\label{eq:1}
f^{\sf D}(\mathcal{F}, \mathcal{S}, B) = \min_{B' \subseteq B}\left\{|B \setminus B'| + \sum_{i=1}^{\ell} x_i(B') \right\}.
\end{equation}

The definition of $f^{\sf D}(\mathcal{F, S}, B)$ comes from the definition of \Dcut{s}: if $B'$ is the subset of $B$ witnessing the value of $f^{\sf D}(\mathcal{F, S}, B)$ then we construct the \Dcut $\mathcal{X} = (B\setminus B', X_1, \ldots, X_\ell)$ such that, for each $i \in [\ell]$, the set $X_i \subseteq V(G_i)$ is a minimum $(S_i, B')$-separator in $G_i$.
Therefore the minimum order of a \Dcut is at most $f^{\sf D}(\mathcal{F, S}, B)$.
On the other hand, when constructing a \Dcut, we first pay one unit for each $b \in B$ that we want to include in our \Dcut (i.e., the value $|B \setminus B'|$) and then, for each $i \in [\ell]$, we pay the minimum cost to separate $S_i$ from the remaining vertices $B' \subseteq B$ that were not included in our \Dcut (i.e., the value $ \sum_{i=1}^{\ell}x_i(B')$).
Thus any \Dcut has order at least $f^{\sf D}(\mathcal{F, S}, B)$ and it follows that the minimum order of a \Dcut is equal to $f^{\sf D}(\mathcal{F, S}, B)$.
We proceed to show that the maximum number of \Dpaths is also equal to $f^{\sf D}(\mathcal{F, S},B)$.
This implies that the maximum number of \Dpaths is equal to the minimum order of a \Dcut, and thus  we obtain a min-max formula.
We begin by showing how to associate paths in $G^{\sf D}$ with \Dpaths, and the size of $(S', B)$-separators in $G^{\sf D}$ with $f^{\sf D}(\mathcal{F, S}, B)$.

\begin{lemma}\label{lemma:paths-association-variation-1}
Given $\mathcal{F, S},$ and $B$, it holds that
\begin{romanenumerate}
  \item for any set of \Dpaths $\mathscr{P}$ there is a set $\mathscr{P'}$ of disjoint $S' \to B$ paths in $G^{\sf D}$ with $|\mathscr{P'}| \geq |\mathscr{P}|$, and
  \item for any set $\mathscr{P'}$ of disjoint $S' \to B$-paths in $G^{\sf D}$ there is a set of \Dpaths $\mathscr{P}$ with $|\mathscr{P}| \geq |\mathscr{P'}|$.
\end{romanenumerate}
\end{lemma}
\begin{proof}
To prove \lipItem{(i)} let $\mathscr{P} = \{P_1, \ldots, P_r\}$ be a set of \Dpaths with partition $\mathcal{P}_1, \ldots, \mathcal{P}_\ell$ as in the definition of \Dpaths.
From each $P_i$ contained in $\mathcal{P}_j$ construct a path $Q'_i$ in $G^{\sf D}$ by following the copies of the vertices of $P_i$ as they appear in $V_j$.
Note that this ensures that each $Q'_i$ has $\source(Q'_i) \in S'_j$. 
Now we construct $P'_i$ from $Q'_i$ by appending to the latter the edge from $\sink(Q'_i)$ to the vertex $v \in B$ such that $\sink(Q'_i)$ is the copy of $v$ in $V_j$. 
Set $\mathscr{P'} = \{P'_1, \ldots, P'_r\}$ and \lipItem{(i)} follows from observing that the defining properties of \Dpaths immediately imply the disjointness of the paths in $\mathscr{P'}$.

To prove \lipItem{(ii)} let $\mathscr{P'} = \{P'_1, \ldots, P'_r\}$ be a set of disjoint $S' \to B$ paths in $G^{\sf D}$.
Since there are no edges between vertices in $B$, every vertex of each $P'_i$, except $\sink(P'_i)$, is contained in some $V_j$ with $j \in [\ell]$.
Thus  we can follow each $P'_i$ in $G^{\sf D}$ to construct a path $P_i$ in some $G_j$ with vertex set $V(P_i) = \{v \in V(G_j) \mid v_j \in V(P'_i) \setminus \{\sink(P'_i)\}\}$.
In other words, $P_i$ contains the vertices of $V(G_j)$ whose copies form $V(P'_i)$ minus the last vertex of this path.
Let $\mathscr{P} = \{P_1, \ldots, P_r\}$ be the set of paths constructed this way and partition $\mathscr{P}$ into sets $\mathcal{P}_i = \{P \in \mathscr{P} \mid \source(P) \in S_i\}$ for each $i \in [\ell]$.
The choice of $\mathscr{P'}$ immediately implies that $\mathscr{P}$ is a set of \Dpaths: since the paths of the first are disjoint, it follows that any distinct pair of paths in $\mathcal{P}_i$ are disjoint in $G_i$ and thus \lipItem{(a)} of \autoref{def:graph-source-sequence} holds.
The choices of the parts $\mathcal{P}_i$ ensures that \lipItem{(b)} of \autoref{def:graph-source-sequence} holds.
Again by the choice of $\mathscr{P'}$ no two paths in $\mathscr{P}$ can share a vertex in $B$ and \lipItem{(1)} of \autoref{def:D-paths-and-D-cuts} holds.
Finally, as $\sink(P'_i) \in B$ for $i \in [\ell]$, by the choice of the edges entering $B$ in $G^{\sf D}$, and observing that $G^{\sf D}[B]$ has no edges, we conclude that $\sink(P) \in B$ for all $P \in \mathscr{P}$ and \lipItem{(2a)} also holds.
\end{proof}

\begin{lemma}\label{lemma:min-separator-variation-1}
Given $\mathcal{F, S}$, and  $B$, the minimum size of an $(S', B)$-separator in $G^{\sf D}$ is equal to $f^{\sf D}(\mathcal{F, S}, B)$.
\end{lemma}
\begin{proof}
Let $X$ be a minimum $(S', B)$-separator in $G^{\sf D}$, let $B' \subseteq B$, and for $i \in [\ell]$ let $X_i$ be an $(S'_i, B')$-separator in $G^{\sf D}$.
Clearly the set $(B\setminus B') \cup X_1 \cup \cdots \cup X_\ell$ contains an $(S', B)$-separator in $G^{\sf D}$, as every $S_i' \to B'$ path in $G^{\sf D}$ is blocked by $X_i$ and every other $S' \to B$ path is intersected by $B \setminus B'$.
Hence $|X| \leq f^{\sf D}(\mathcal{F, S}, B)$.

By contradiction, assume that $|X| < f^{\sf D}(\mathcal{F, S}, B)$. 
Then  $f^{\sf D}(\mathcal{F, S}, B) > |X| = |B \cap X| + \sum_{i=i}^{\ell} |X\cap V_i|$.
Defining $X_i = X \cap V_i$ for each $i \in [\ell]$ and $B^* = B \setminus X$, we conclude that $f^{\sf D}(\mathcal{F, S}, B) > |X| = |B \setminus B^*| + \sum_{i=1}^{\ell} |X_i|$.
Since each $X_i$ is a $(S_i, B^*)$-separator and $B^* \subseteq B$, this contradicts the minimality of $f^{\sf D}(\mathcal{F, S}, B)$ and thus $|X|$ has to be equal to $f^{\sf D}(\mathcal{F, S}, B)$, and the result follows.
\end{proof}

\begin{theorem}\label{theorem:min-max-variation-1}
Given $\mathcal{F}$, $\mathcal{S}$, and $B$, the maximum size of a set of \Dpaths is equal to $f^{\sf D}(\mathcal{F}, \mathcal{S}, B)$.
\end{theorem}
\begin{proof}
The result easily follows by applying \autoref{thm:Menger} to $G^{\sf D}$ and \cref{lemma:paths-association-variation-1,lemma:min-separator-variation-1}.
Indeed, assume that $\mathscr{P}$ is a maximum set of \Dpaths.
By item \lipItem{(i)} of \autoref{lemma:paths-association-variation-1} we find a maximum set $\mathscr{P}'$ of disjoint $S' \to B$ paths in $G^{\sf D}$ with $|\mathscr{P'}| \geq |\mathscr{P}|$.
By \autoref{thm:Menger} we know that that $|\mathscr{P}|$ is equal to the size of a minimum $(S', B)$-separator $X$ in $G^{\sf D}$, and by \autoref{lemma:min-separator-variation-1} we know that $|X| = f^{\sf D}(\mathcal{F, S}, B)$.
Plugging together those statements we obtain $|\mathscr{P}| \leq |\mathscr{P}'| = |X| = f^{\sf D}(\mathcal{F, S}, B)$.
Similarly, applying \autoref{thm:Menger}, item \lipItem{(ii)} of \autoref{lemma:paths-association-variation-1}, and \autoref{lemma:min-separator-variation-1} we conclude that any maximum set of \Dpaths $\mathscr{P}$ satisfies $|\mathscr{P}| \geq f^{\sf D}(\mathcal{F, S}, B)$, and the result follows.
\end{proof}

In the paragraph after \autoref{eq:1} we argued that the minimum order of a \Dcut is equal to the value $f^{\sf D}(\mathcal{F, S}, B)$.
With this observation and applying \cref{thm:Menger,theorem:min-max-variation-1} we immediately obtain \autoref{theo:min-max-statement-var-1}.
Observe that the running time follows from the fact that $|V(G^D)| = \ell \cdot n^* + |B|$, where $n^*$ is the maximum number of vertices of a digraph in $\mathcal{F}$, and from the running time mentioned in the statement of \autoref{thm:Menger}, which is applied to $G^{\sf D}$.

\medskip
\noindent\textbf{Proof for \Tpaths and \Tcut{s}.}
In this part, we assume  that $\mathcal{B} = \{B_1, \ldots, B_r\}$ is a family of subsets of $\bigcup_{i = 1}^{\ell}V(G_i)$.
Similarly to the previous case, we begin with $G^{\sf T} = G(\mathcal{F})$.
Now, set $V^* = \bigcup_{i=1}^{r}B_i$, add to $G^{\sf T}$ every vertex appearing in $V^*$, and add an edge from $u \in \bigcup_{i=1}^{\ell}V_i$ to $v \in V^*$ if and only if $u$ is a copy of $v$.
Define $S'_i = \{s_i \mid s \in S_i\}$ and $S' = \bigcup_{i = 1}^{\ell}S'_i$.
To finish the construction, add to $G^{\sf T}$ a set of vertices $B = \{b_1, \ldots, b_r\}$ and one edge from $v \in V^*$ to $b_i$ if and only if $v \in B_i$.
Thus the vertices $\{b_1, \ldots, b_r\}$ induce an independent set in $G^{\sf T}$ and each of these vertices can be accessed only from $V^*$ in $G^{\sf T}$.
See \autoref{fig:construction-T-paths} for an example of this construction.

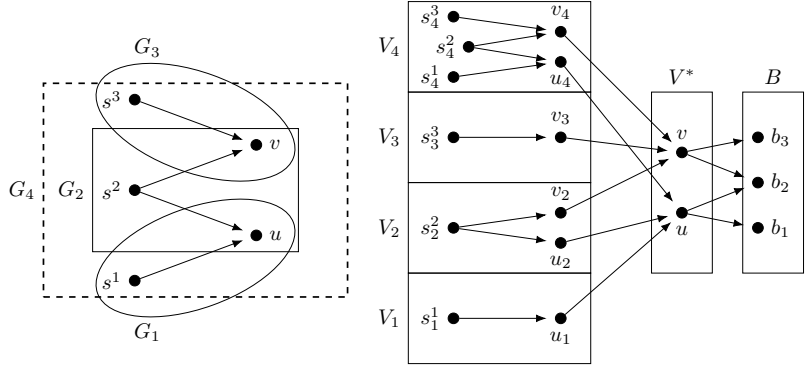
\begin{figure}[h]
\centering
\scalebox{.8}{
\begin{tikzpicture}
\node[draw, rectangle, minimum width=3cm, minimum height=1.5cm, label=180:{$V_1$}] (d1) at (0,0) {};
\node[blackvertex, label=180:{$s^1_1$}] (s1) at ($(d1.west) + (.75, 0)$) {};
\node[draw, rectangle, minimum width=3cm, minimum height=1.5cm, label=180:{$V_2$}] (d2) at (0,1.5) {};
\node[blackvertex, label=180:{$s^2_2$}] (s2) at ($(d2.west) + (.75, 0)$) {};
\node[draw, rectangle, minimum width=3cm, minimum height=1.5cm, label=180:{$V_3$}] (d3) at (0,3) {};
\node[blackvertex, label=180:{$s^3_3$}] (s3) at ($(d3.west) + (.75, 0)$) {};
\node[draw, rectangle, minimum width=3cm, minimum height=1.5cm, label=180:{$V_4$}] (d4) at (0,4.5) {};
\node[blackvertex, label=180:{$s^1_4$}] (s14) at ($(d4.west) + (.75, -0.5)$) {};
\node[blackvertex, label=180:{$s^2_4$}] (s24) at ($(d4.west) + (1, 0)$) {};
\node[blackvertex, label=180:{$s^3_4$}] (s34) at ($(d4.west) + (.75, 0.5)$) {};

\begin{scope}[yshift=.75cm]
\node[draw, rectangle, label=90:{$B$}, minimum width=1cm, minimum height = 3cm] (B) at (4.5,1.5) {};

\node[draw, rectangle, label=90:{$V^*$}, minimum width=1cm, minimum height = 3cm] (Vs) at (3,1.5) {};

\node[blackvertex, label=-90:{$u_1$}] (u1) at ($(d1.east) + (-.5, 0)$) {};

\node[blackvertex, label=-90:{$u_2$}] (u2) at ($(d2.east) + (-.5, -0.25)$) {};
\node[blackvertex, label=90:{$v_2$}] (v2) at ($(d2.east) + (-.5, 0.25)$) {};

\node[blackvertex, label=90:{$v_3$}] (v3) at ($(d3.east) + (-.5, 0)$) {};

\node[blackvertex, label=0:{$b_1$}] (b1) at ($(B) + (-0.25, -.75)$) {};
\node[blackvertex, label=0:{$b_2$}] (b2) at ($(B) + (-0.25, 0)$) {};
\node[blackvertex, label=0:{$b_3$}] (b3) at ($(B) + (-0.25, .75)$) {};

\node[blackvertex, label=-90:{$u$}] (u) at ($(Vs) + (0, -.5)$) {};
\node[blackvertex, label=90:{$v$}] (v) at ($(Vs) + (0, .5)$) {};

\node[blackvertex, label=-90:{$u_4$}] (u4) at ($(d4.east) + (-.5, -0.25)$) {};
\node[blackvertex, label=90:{$v_4$}] (v4) at ($(d4.east) + (-.5, 0.25)$) {};
\end{scope}

\draw[arrow] (s1) -- (u1);

\draw[arrow] (s2) -- (u2);
\draw[arrow] (s2) -- (v2);

\draw[arrow] (s3) -- (v3);

\draw [arrow] (s14) -- (u4);
\draw [arrow] (s24) -- (u4);
\draw [arrow] (s24) -- (v4);
\draw [arrow] (s34) -- (v4);

\draw[arrow] (v4) -- (v);
\draw[arrow] (u4) -- (u);

\draw[arrow] (v3) -- (v);

\draw[arrow] (v2) -- (v);
\draw[arrow] (u2) -- (u);

\draw[arrow] (u1) -- (u);

\draw[arrow] (u) -- (b1);
\draw[arrow] (u) -- (b2);
\draw[arrow] (v) -- (b2);
\draw[arrow] (v) -- (b3);

\begin{scope}[xshift=-6cm, yshift = .625cm]
\node[blackvertex, label=180:{$s^1$}] (s1) at (0,0) {};
\node[blackvertex, label=180:{$s^2$}] (s2) at (0,1.5) {};
\node[blackvertex, label=180:{$s^3$}] (s3) at (0,3) {};

\node[blackvertex, label=0:{$u$}] (u) at ($(s2) + (2,-.75)$) {};
\node[blackvertex, label=0:{$v$}] (v) at ($(s2) + (2,.75)$) {};
\node[draw, ellipse, fit = (s3)(v), rotate = -20, label=135:{$G_3$}] {};
\node[draw, ellipse, fit = (s1)(u), rotate = 20, label=-135:{$G_1$}] {};
\node[draw, rectangle, fit = (s2)(v)(u), inner xsep = 17pt, inner ysep=5pt, label=180:{$G_2$}] {};
\node[draw, rectangle, dashed, thick, fit = (s1)(s3)(v)(u), inner xsep = 40pt, inner ysep=5pt, label=180:{$G_4$}] {};
\draw[arrow] (s1) -- (u);
\draw[arrow] (s2) -- (u);
\draw[arrow] (s2) -- (v);
\draw[arrow] (s3) -- (v);

\end{scope}
\end{tikzpicture}}%
\caption{On the left, example graphs $G_1$, $G_2$, $G_3$, and $G_4$ with $S_i = \{s^i\}$ for $i \in [3]$ and $S_4 = \{s_1, s_2, s_3\}$. On the right, the resulting construction of $G^{\sf T}$ with $B_1 = \{u\}, B_2 = \{u,v\}$, and $B_3 = \{v\}$.}
\label{fig:construction-T-paths}
\end{figure}

For the corresponding min-max formula, let
\begin{equation}\label{eq:2}
f^{\sf T}(\mathcal{F, S, B}) = \min_{\mathcal{B'}\subseteq \mathcal{B}}\left\{|\mathcal{B \setminus B'}| + f^{\sf D}\left(\mathcal{F,S},\cupall\mathcal{B}'\right) \right\}.
\end{equation}

The definition of this function comes from observing the shape of \Tcut{s}.
As we did for \Dcut{s}, we give an intuition on how to associate the value $f^{\sf T}(\mathcal{F, S, B})$ with \Tcut{s} by thinking of the order of a \Tcut as the cost we pay to construct it.
Since we have no control over the size of each $A \in \mathcal{B}$, we cannot simply count vertices to define this cost since it is possible, for example, that some $A \in \mathcal{B}$ contains every vertex of every $G_i$.
This message is implicit in the definition of $f^{\sf T}(\mathcal{F, S, B})$: when looking at the left side of \autoref{eq:2} we see that we are allowed to choose a $\mathcal{B}'\subseteq \mathcal{B}$ and pay one unit for each $A \in \mathcal{B}\setminus \mathcal{B}'$ we want to {\sl ignore} as a destination for the \Rpaths (i.e., the value $ |\mathcal{B} \setminus \mathcal{B}'|$).
Then, we pay the cost to construct a \Dcut with respect to $\mathcal{F, S}$, and $\cupall\mathcal{B}'$.
The intuition for this is visible in the auxiliary digraph (see \autoref{fig:construction-T-paths}): by ``compressing'' each $B_i \in \mathcal{B}$ to a single vertex $b_i$ having as in-neighbors every $v \in B_i$ in $G^{\sf T}$ we ensure that, for any set of \Dpaths $\mathscr{Q}$, the set $\{s^{P} \mid P \in \mathscr{Q}\}$ forms a partial transversal of $\mathcal{B}$.

From the above discussion, one can verify that the minimum order of a \Tcut is at least $f^{\sf T}(\mathcal{F, S, B})$.
To see the that $f^{\sf T}(\mathcal{F, S, B})$ is at least the minimum order of a \Tcut, it suffices to construct a \Tcut from a $\mathcal{B}' \subseteq \mathcal{B}$ witnessing the value of the function, and thus those two values are equal.
We show that the maximum number of \Tpaths is also equal to $f^{\sf T}(\mathcal{F, S, B})$, and use this fact to obtain a min-max formula.
We begin by showing how to associate paths in $G^{\sf T}$ with \Tpaths, and the size of $(S', B)$-separators in $G^{\sf T}$ with $f^{\sf T}(\mathcal{F, S}, B)$.

\begin{lemma}\label{lemma:paths-association-variation-2}
Given $\mathcal{F, S}$, and $\mathcal{B}$, it holds that
\begin{romanenumerate}
  \item for any set of \Tpaths $\mathscr{P}$ there is a set $\mathscr{P'}$ of disjoint $S' \to B$ paths in $G^{\sf T}$ with $|\mathscr{P'}| \geq |\mathscr{P}|$, and
  \item for any set $\mathscr{P'}$ of disjoint $S' \to B$ paths in $G^{\sf T}$ there is a set of \Tpaths $\mathscr{P}$ with $|\mathscr{P}| \geq |\mathscr{P'}|$.
\end{romanenumerate}
\end{lemma}
\begin{proof}
To prove \lipItem{(i)}, let $\mathscr{P} = \{P_1, \ldots, P_q\}$ be a set of \Tpaths and consider a partition $\mathcal{P}_1, \ldots, \mathcal{P}_\ell$ witnessing the defining properties of \Tpaths. 
From each path $P_i$ in $\mathcal{P}_j$ for some $j \in [\ell]$ construct a path $Q_i$ in $G^{\sf T}$ by taking the copies of vertices in $V(P_i)$ appearing in $V_j$.
This ensures that $\source(Q_i) \in S'_j$ whenever $P_i$ is contained in $G_j$, and that the paths $Q_1, \ldots, Q_q$ are disjoint since the paths in each part of $\mathcal{P}_j$ are disjoint.
Now, each $\sink(Q_i)$ is a copy of some vertex $v \in V^*$ and we construct $Q'_i$ appending to $Q_i$ the edge from $\sink(Q_i)$ to $v$ in $G^{\sf T}$.
Observe that \lipItem{(1)} of \autoref{def:t-paths-and-t-cuts} implies that the paths remain disjoint in $G^{\sf T}$.
Finally, for each $i \in [q]$, \lipItem{(2b)} of \autoref{def:t-paths-and-t-cuts} implies that the edges of $G^{\sf T}$ of the form $(\sink(Q_i), b_j)$, for each pair $i,j$ such that $\sink(P_i) \in B_j$, contain a matching of size $q$ in $G^{\sf T}$.
We construct the paths $P'_i$ by appropriately appending to each $Q_i$ the edges of this matching, set $\mathscr{P'} = \{P'_1, \ldots, P'_q\}$, and \lipItem{(i)} follows.

To prove \lipItem{(ii)}, let $\mathscr{P}' = \{P'_1, \ldots, P'_q\}$ be a set of disjoint $S'\to B$ paths in $G^{\sf T}$.
For each $i \in [q]$ we construct the path $P_i$ by taking the vertices of $V(G_j)$ whose copies form $V(P'_i) \cap V_j$ and set $\mathscr{P} = \{P_1, \ldots, P_q\}$.
For $i \in [\ell]$ let $\mathcal{P}_i = \{P \in \mathscr{P} \mid \source(P) \in S_i\}$.
Naturally \lipItem{(a)} and \lipItem{(b)} of \autoref{def:graph-source-sequence} hold with relation to $\mathscr{P}$ by the disjointness of paths in $\mathscr{P}'$ and since each $P'_i$ is an $S' \to B$ path.
Notice that no two paths in $\mathscr{P}'$ can share a vertex in $B$.
This is true because those paths are disjoint, and thus they do not intersect in $V^*$, and they all end in $B$.
Now, since each $P'_i$ ends in a distinct $b_j \in B$ and $B$ induces an independent set in $G^{\sf T}$, the choice of the edges from the sets $V_j$ to $V^*$ in $G^{\sf T}$ implies that the set $X \subseteq V^*$ whose copies form $\{\sink(P_1), \ldots, \sink(P_q)\}$ is a partial transversal of $\mathcal{B}$.
Together with the disjointness of paths of $\mathscr{P}'$, this implies that \lipItem{(1)} and \lipItem{(2b)} of \autoref{def:t-paths-and-t-cuts} hold with relation to $\mathscr{P}$, and \lipItem{(ii)} follows.
\end{proof}

\begin{lemma}\label{lemma:min-separator-variation-2}
Given $\mathcal{F, S}$, and $\mathcal{B}$, the minimum size of an $(S', B)$-separator in $G^{\sf T}$ is equal to $f^{\sf T}(\mathcal{F, S, B})$.
\end{lemma}
\begin{proof}
Let $X$ be a minimum $(S', B)$-separator in $G^{\sf T}$.
For any $B' \subseteq B$, let $\mathcal{B}(B')$ contain exactly the sets $A \in \mathcal{B}$ such that $A \subseteq N^-_{G^{\sf T}}(v)$ for some $v \in B'$ (notice that this choice implies that $|B \setminus B'| = |\mathcal{B} \setminus \mathcal{B}(B')|$).
Define $T(B') = \cupall\mathcal{B}(B')$.

By \autoref{theo:min-max-statement-var-1} we know that any minimum \Dcut $\mathcal{X}(B')$ for $(\mathcal{F, S}, T(B'))$ has order equal to $f^{\sf D}(\mathcal{F,S}, T(B'))$.
Since $(B \setminus B') \cup (\cupall \mathcal{X}(B'))$ is an $(S', B)$-separator in $G^{\sf T}$, the choice of $X$ implies that, for any $B' \subseteq B$,
\begin{align*}
|X| &\leq |B \setminus B'| + \order(X(B'))\\
 &= |\mathcal{B} \setminus \mathcal{B}(B')| + \order(X(B'))\\
 &= |\mathcal{B} \setminus \mathcal{B}(B')| + f^{\sf D}(\mathcal{F, S}, T(B')),
\end{align*}
and hence we conclude that $|X| \leq f^{\sf T}(\mathcal{F, S, B})$.

By contradiction, assume now that $|X| < f^{\sf T}(\mathcal{F, S, B})$.
Define $B^* = B \setminus X$, let $\mathcal{B}^*$ contain all $A \in \mathcal{B}$ such that $A \subseteq N^-_{G^{\sf T}}(v)$ for some $v \in B^*$, and $T = \cupall\mathcal{B^*}$.
This choice immediately implies that $|B \setminus B^*| = |\mathcal{B} \setminus \mathcal{B}^*|$.
Consider now \Dpaths/\Dcut{s} with respect to $\mathcal{F, S}$, and $T$.
Since $X$ is a minimum $(S', B)$-separator, we know that $X \setminus B$ is a minimum $(S', T)$-separator, and thus by \autoref{lemma:min-separator-variation-1} we get $|X \setminus B| = f^{\sf D}(\mathcal{F, S}, T)$.
By our assumption that $|X| < f^{\sf T}(\mathcal{F, S, B})$ we get that

\begin{align*}
f^{\sf T}(\mathcal{F, S, B}) &> |X| = |X\cap B| + |X \setminus B| = |B \setminus B^*| + |X \setminus B|\\
&= |\mathcal{B} \setminus \mathcal{B}^*| + |X \setminus B|\\
&= |\mathcal{B} \setminus \mathcal{B}^*| + f^{\sf D}(\mathcal{F, S}, T),
\end{align*}
contradicting the minimality of $f^{\sf T}(\mathcal{F, S, B})$ and thus we conclude that $|X| \geq f^{\sf T}(\mathcal{F, S, B})$, and the result follows.
\end{proof}

\begin{theorem}\label{theorem:min-max-variation-2}
Given $\mathcal{F}$, $\mathcal{S}$, and $\mathcal{B}$, the maximum size of a set of \Tpaths is equal to $f^{\sf T}(\mathcal{F}, \mathcal{S}, \mathcal{B})$.
\end{theorem}
\begin{proof}
Let $\mathscr{P}$ be a maximum set of \Tpaths.
By item \lipItem{(i)} of \autoref{lemma:paths-association-variation-2} there is a maximum set $\mathscr{P}'$ of disjoint $S' \to B$ paths in $G^{\sf T}$ with $|\mathscr{P}'| \geq |\mathscr{P}|$ and by \autoref{thm:Menger} it holds that $|\mathscr{P}'| = |X|$ where $X$ is a minimum $(S', B)$-separator in $G^{\sf T}$.
By \autoref{lemma:min-separator-variation-2} we know that $|X| = f^{\sf T}(\mathcal{F, S, B})$ and hence $|\mathscr{P}| \leq |\mathscr{P}'| = |X| = f^{\sf T}(\mathcal{F, S, B})$.
Similarly, applying \autoref{thm:Menger}, item \lipItem{(ii)} of \autoref{lemma:paths-association-variation-2}, and \autoref{lemma:min-separator-variation-2} we conclude that any maximum set of \Tpaths has size bounded from below by $f^{\sf T}(\mathcal{F, S, B})$ and the result follows.
\end{proof}

As mentioned in the paragraph after \autoref{eq:2}, by choosing $\mathcal{B}' \subseteq \mathcal{B}$ minimizing the order of a \Tcut $(\mathcal{B'}, X)$, it clearly holds that the minimum order of a \Tcut is equal to $f^{\sf T}(\mathcal{F, S, B})$.
Thus, applying \cref{thm:Menger,theorem:min-max-variation-2} and observing the bound on the running time of \autoref{thm:Menger}, which we apply to $G^{\sf T}$, we immediately obtain \autoref{theo:min-max-statement-var-2}.

\medskip
\noindent\textbf{Proof for \Rpaths and \Rcut{s}.}
In this part, we assume again that $\mathcal{B} = \{B_1, \ldots, B_r\}$ is a family of subsets of $\bigcup_{i = 1}^{\ell}V(G_i)$.
Start with $G^{\sf R} = G(\mathcal{F})$, add to $G^{\sf R}$ a set of vertices $\{b_1, \ldots, b_r\}$ and, for each $u \in \bigcup_{i = 1}^{\ell}V(G_i)$ and $j \in [r]$, add an edge from each copy of $u$ to $b_j$ in $G^{\sf R}$ if $u \in B_j$.
Define $S'_i = \{s_i \mid s \in S_i\}$ and $S' = \bigcup_{i = 1}^{\ell}S'_i$.
See \autoref{fig:construction-R-paths} for an illustration of this construction.

\begin{figure}[h]
\centering
\scalebox{.8}{
\begin{tikzpicture}
\node[draw, rectangle, minimum width=3cm, minimum height=1.5cm, label=180:{$V_1$}] (d1) at (0,0) {};
\node[blackvertex, label=180:{$s_1$}] (s1) at ($(d1.west) + (.75, 0)$) {};
\node[draw, rectangle, minimum width=3cm, minimum height=1.5cm, label=180:{$V_2$}] (d2) at (0,1.5) {};
\node[blackvertex, label=180:{$s_2$}] (s2) at ($(d2.west) + (.75, 0)$) {};
\node[draw, rectangle, minimum width=3cm, minimum height=1.5cm, label=180:{$V_3$}] (d3) at (0,3) {};
\node[blackvertex, label=180:{$s_3$}] (s3) at ($(d3.west) + (.75, 0)$) {};

\node[draw, rectangle, label=90:{$B$}, minimum width=1cm, minimum height = 3cm] (B) at (3,1.5) {};


\node[blackvertex, label=-90:{$u_1$}] (u1) at ($(d1.east) + (-.5, 0)$) {};

\node[blackvertex, label=-90:{$u_2$}] (u2) at ($(d2.east) + (-.5, -0.25)$) {};
\node[blackvertex, label=90:{$v_2$}] (v2) at ($(d2.east) + (-.5, 0.25)$) {};

\node[blackvertex, label=90:{$v_3$}] (v3) at ($(d3.east) + (-.5, 0)$) {};

\node[blackvertex, label=0:{$b_1$}] (b1) at ($(B) + (-0.25, -.75)$) {};
\node[blackvertex, label=0:{$b_2$}] (b2) at ($(B) + (-0.25, 0)$) {};
\node[blackvertex, label=0:{$b_3$}] (b3) at ($(B) + (-0.25, .75)$) {};


\draw[arrow] (s1) -- (u1);

\draw[arrow] (s2) -- (u2);
\draw[arrow] (s2) -- (v2);

\draw[arrow] (s3) -- (v3);




\draw[arrow] (u1) -- (b1);
\draw[arrow] (u2) -- (b1);
\draw[arrow] (u1) -- (b2);
\draw[arrow] (u2) -- (b2);
\draw[arrow] (v2) -- (b2);
\draw[arrow] (v3) -- (b2);
\draw[arrow] (v2) -- (b3);
\draw[arrow] (v3) -- (b3);

\begin{scope}[xshift=-6cm]
\node[blackvertex, label=180:{$s_1$}] (s1) at (0,0) {};
\node[blackvertex, label=180:{$s_2$}] (s2) at (0,1.5) {};
\node[blackvertex, label=180:{$s_3$}] (s3) at (0,3) {};

\node[blackvertex, label=0:{$u$}] (u) at ($(s2) + (2,-.75)$) {};
\node[blackvertex, label=0:{$v$}] (v) at ($(s2) + (2,.75)$) {};
\node[draw, ellipse, fit = (s3)(v), rotate = -20, label=90:{$G_3$}] {};
\node[draw, ellipse, fit = (s1)(u), rotate = 20, label=-90:{$G_1$}] {};
\node[draw, rectangle, fit = (s2)(v)(u), inner xsep = 17pt, inner ysep=5pt, label=180:{$G_2$}] {};
\draw[arrow] (s1) -- (u);
\draw[arrow] (s2) -- (u);
\draw[arrow] (s2) -- (v);
\draw[arrow] (s3) -- (v);
\end{scope}
\end{tikzpicture}}%
\caption{On the left, example graphs $G_1$, $G_2$, and $G_3$. On the right, the resulting construction of $G^{\sf R}$ with $B_1 = \{u\}, B_2 = \{u,v\}$, and $B_3 = \{v\}$.}
\label{fig:construction-R-paths}
\end{figure}
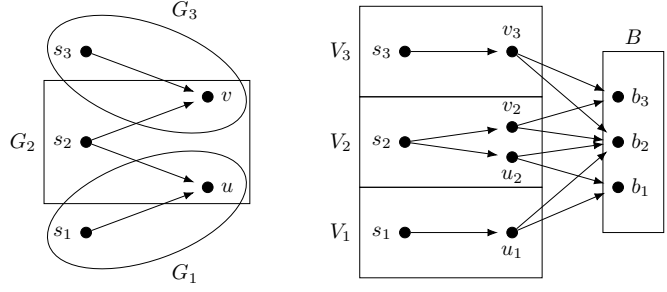

For $i \in [\ell]$ and $\mathcal{B}' \subseteq \mathcal{B}$ let $x_i(\mathcal{B}')$ be the size of a minimum $(S_i, \cupall\mathcal{B'})$-separator in $G_i$, and
\begin{equation}\label{eq:min-max-formula-variation-3}
f^{\sf R}(\mathcal{F, S, B}) = \min_{\mathcal{B}' \subseteq \mathcal{B}}\left\{|\mathcal{B} \setminus \mathcal{B}'| + \sum_{i=1}^{\ell} x_i(\mathcal{B}') \right\}.
\end{equation}

Similarly to \Tcut{s} the definition of $f^{\sf R}(\mathcal{F, S, B})$ follows the shape of \Rcut{s}.
As it was the case with \Tcut{s}, we cannot construct \Rcut{s} by paying one unit for each vertex of $\bigcup_{i=1}^{k}V(G_i)$ that we want to include in the cut, since its possible that some $A \in \mathcal{B}$ contains every vertex of every digraph in $\mathcal{F}$.
Thus, we are again allowed to choose a $\mathcal{B'} \subseteq \mathcal{B}$ and pay one unit for each $A \in \mathcal{B}\setminus \mathcal{B}$ that we ignore as destinations for the \Rpaths (i.e., the value $|\mathcal{B} \setminus \mathcal{B}'|$).
Then, for each $i \in [k]$ we pay the minimum cost to separate $S_i$ from every vertex in $\cupall\mathcal{B}'$ (i.e., the value $\sum_{i=1}^{k}x_i(\mathcal{B}')$).

As in the previous cases, the above discussion and a simple observation that every $\mathcal{B}'\subseteq \mathcal{B}$ witnessing the value of $f^{\sf R}(\mathcal{F, S, B})$ can be used to construct an \Rcut with this order implies that the minimum order of an \Rcut is equal to $f^{\sf R}(\mathcal{F, S, B})$.
Following the proofs for \Dpaths/\Dcut{s} and \Tpaths/\Tcut{s}, to obtain the desired min-max formula we show that the maximum number of \Tpaths is equal to $f^{\sf R}(\mathcal{F, S, B})$.
We first show how to associate paths in $G^{\sf R}$ with \Rpaths, and the size of $(S', B)$-separators in $G^{\sf R}$ with $f^{\sf R}(\mathcal{F, S}, B)$.

\begin{lemma}\label{lemma:paths-association-variation-3}
Given $\mathcal{F, S}$, and $\mathcal{B}$, it holds that
\begin{romanenumerate}
  \item for any set of \Rpaths $\mathscr{P}$ there is a set $\mathscr{P'}$ of disjoint $S' \to B$ paths in $G^{\sf R}$ with $|\mathscr{P'}| \geq |\mathscr{P}|$, and
  \item for any set $\mathscr{P'}$ of disjoint $S' \to B$ paths in $G^{\sf R}$ there is a set of \Rpaths $\mathscr{P}$ with $|\mathscr{P}| \geq |\mathscr{P'}|$.
\end{romanenumerate}
\end{lemma}
\begin{proof}
To prove \lipItem{(i)}, let $\mathscr{P} = \{P_1, \ldots, P_q\}$ be a set of \Rpaths and consider a partition $\mathcal{P}_1, \ldots, \mathcal{P}_\ell$ witnessing the defining properties of \Rpaths. 
From each path $P_i$ in $\mathcal{P}_j$ for some $j \in [\ell]$ construct a path $Q_i$ in $G^{\sf R}$ by taking the copies of vertices in $V(P_i)$ appearing in $V_j$.
This ensures that $\source(Q_i) \in S'_j$ whenever $P_i$ is contained in $G_j$, and that the paths $Q_1, \ldots, Q_q$ are disjoint since the paths in each part of $\mathcal{P}_j$ are disjoint.
Now, let $\mathcal{B}^* \subseteq \mathcal{B}$ and $h : \mathscr{P} \to \mathcal{B}^*$ be a bijective mapping as in \lipItem{(1c)} of \autoref{def:r-paths-and-r-cuts}.
We construct the paths $P'_i$ appending to each $Q_i$ the edge in $G^R$ from $\sink(P_i)$ to the vertex $b \in B$ that is associated with the bag $h(P_i) \in \mathcal{B}^*$.
This edge is guaranteed to exist because $\sink(P_i) \in h(P_i)$.

To prove \lipItem{(ii)}, let $\mathscr{P}' = \{P'_1, \ldots, P'_q\}$ be a set of disjoint $S'\to B$ paths in $G^{\sf R}$.
For each $i \in [q]$ we construct the path $P_i$ by taking the vertices of $V(G_j)$ whose copies form $V(P'_i) \cap V_j$, and set $\mathscr{P} = \{P_1, \ldots, P_q\}$.
For $i \in [\ell]$ let $\mathcal{P}_i = \{P \in \mathscr{P} \mid \source(P) \in S_i\}$.
Naturally \lipItem{(a)} and \lipItem{(b)} of \autoref{def:graph-source-sequence} hold with relation to $\mathscr{P}$ by the disjointness of paths in $\mathscr{P}'$ and since each $P'_i$ is an $S' \to B$ path.
Notice that since those paths are disjoint, no two paths in $\mathscr{P}'$ can share a vertex in $B$.
Now, since each $P'_i$ ends in a distinct $b_j \in B$ and $B$ induces an independent set in $G^{\sf R}$, the choice of the edges from $\bigcup_{j = 1}^{\ell}V_j$ to $B$ in $G^{\sf R}$ implies that the last edge of each path $P'_i$ is of the form $(v, b_j)$ for some $j \in [r]$.
Thus we can construct the desired mapping $h: \mathscr{P} \to \mathcal{B}^*$, with $\mathcal{B}^* \subseteq \mathcal{B}$, by setting for each $i \in [q]$, $h(P_i) = B_j$ if $b_j$ is the last vertex of $P'_i$. Thus  \lipItem{(1c)} of \autoref{def:r-paths-and-r-cuts} holds and the result follows.
\end{proof}

The proof of \autoref{lemma:paths-association-variation-3} is similar to the proof of \autoref{lemma:paths-association-variation-2}.
The only difference is that the set $V^*$ is not present in $G^{\sf R}$ and thus the edges reaching the vertices in $B$ are directly from vertices in $\bigcup_{i = 1}^{\ell} V_i$.
Hence the last vertices of the paths in $\mathscr{P}'$ are not guaranteed to be disjoint.
This is not an issue since in the definition of \Rpaths this property (item \lipItem{(1)} in \cref{def:D-paths-and-D-cuts,def:t-paths-and-t-cuts}) is not required.

\begin{lemma}\label{lemma:min-separator-variation-3}
Given $\mathcal{F, S}$, and $\mathcal{B}$, the minimum size of an $(S', B)$-separator in $G^{\sf R}$ is equal to $f^{\sf R}(\mathcal{F, S}, B)$.
\end{lemma}
\begin{proof}
Again, the proof follows similarly to the proof of \autoref{lemma:min-separator-variation-2}.
Let $X$ be a minimum $(S', B)$-separator in $G^{\sf R}$.
For any $B' \subseteq B$ let $\mathcal{B}(B')$ contain exactly the sets $A \in \mathcal{B}$ such that $A \subseteq N^-_{G^{\sf R}}(v)$ for some $v \in B'$.
This immediately implies that $|B \setminus B'| = |\mathcal{B} \setminus \mathcal{B}(B')|$.
Define $T(B') = \cupall\mathcal{B}(B')$, and for $i \in [\ell]$ let $X_i(B')$ be a minimum $(S'_i, T(B'))$-separator in $G^{\sf R}$.
Hence every $X_i(B')$ is contained in $V_i$ (we remind the reader of \autoref{def:disjoint-copies-digraph}).
Clearly $(B\setminus B') \cup (\bigcup_{i=1}^{\ell}X_i(B'))$ is an $(S', B)$-separator in $G^{\sf R}$ since every $S' \to B'$ is intersected by some $X_i(B')$ (notice that no path can have an internal vertex in $B$).
Thus for any $B' \subseteq B$ it holds that
\begin{align*}
|X| &\leq |B \setminus B'| + \left| \bigcup_{i=1}^{\ell}X_i(B') \right|\\
  &= |B \setminus B'| + \sum_{i=1}^{\ell}x_i(\mathcal{B}(B'))\\
  &= |\mathcal{B} \setminus \mathcal{B}(B')| + \sum_{i=1}^{\ell}x_i(\mathcal{B}(B')),
\end{align*}
and we conclude that $|X| \leq f^{\sf R}(\mathcal{F, S, B})$ (the value of $x_i(\mathcal{B}(B'))$ as defined right before \autoref{eq:min-max-formula-variation-3}).

By contradiction, assume now that $|X| < f^{\sf R}(\mathcal{F, S, B})$. 
Define $B^* = B \setminus X$, let $\mathcal{B}^*$ contain all $A \in \mathcal{B}$ such that $A \subseteq N^-_{G^{\sf R}}(v)$ for some $v \in B^*$, and define $T = \cupall\mathcal{B}^*$.
We remark that the choice of $B^*$ immediately implies that $|B \setminus B^*| = |\mathcal{B} \setminus \mathcal{B}^*|$.
Clearly $X \setminus B$ is a minimum $(S', B^*)$-separator in $G^{\sf R}$, for otherwise there would be an $(S', B)$-separator smaller than $X$.
Thus by our assumption that $|X| < f^{\sf R}(\mathcal{F, S, B})$ we obtain
\begin{align*}
f^{\sf R}(\mathcal{F, S, B}) &> |X| = |X \cap B| + |X \setminus B| = |B \setminus B^*| + |X \setminus B| = |\mathcal{B} \setminus \mathcal{B}^*| + |X \setminus B|\\
  &= |\mathcal{B} \setminus \mathcal{B}^*| + \sum_{i=1}^{\ell}|X \cap V_i|\\
  &= |\mathcal{B} \setminus \mathcal{B}^*| + \sum_{i=1}^{\ell}x_i(\mathcal{B}^*),
\end{align*}
by observing that each $x_i(\mathcal{B}^*)$ is equal to the size of a minimum $(S'_i, T)$-separator in $G^{\sf R}$.
Thus we conclude that $|X| \geq f^{\sf R}(\mathcal{F, S, B})$ and the result follows.
\end{proof}

\begin{theorem}\label{theorem:min-max-variation-3}
Given $\mathcal{F}$, $\mathcal{S}$, and $\mathcal{B}$, the maximum size of a set of \Rpaths set is equal to $f^{\sf R}(\mathcal{F}, \mathcal{S}, \mathcal{B})$.
\end{theorem}
\begin{proof}
Let $\mathscr{P}$ be a maximum set of \Rpaths.
By item \lipItem{(i)} of \autoref{lemma:paths-association-variation-3} there is a maximum set $\mathscr{P}'$ of disjoint $S' \to B$ paths in $G^{\sf R}$ with $|\mathscr{P}'| \geq |\mathscr{P}|$ and by \autoref{thm:Menger} it holds that $|\mathscr{P}'| = |X|$ where $X$ is a minimum $(S', B)$-separator in $G^{\sf R}$.
By \autoref{lemma:min-separator-variation-3} we know that $|X| = f^{\sf R}(\mathcal{F, S, B})$ and hence $|\mathscr{P}| \leq |\mathscr{P}'| = |X| = f^{\sf R}(\mathcal{F, S, B})$.
Similarly, applying \autoref{thm:Menger}, item \lipItem{(ii)} of \autoref{lemma:paths-association-variation-3}, and \autoref{lemma:min-separator-variation-3} we conclude that any maximum set of \Rpaths has size bounded from below by $f^{\sf R}(\mathcal{F, S, B})$, and the result follows.
\end{proof}

Finally, applying \cref{thm:Menger,theorem:min-max-variation-3} and since the minimum order of an \Rcut is equal to $f^{\sf R}(\mathcal{F, S, B})$, we obtain \autoref{theo:min-max-statement-var-3}.
The running time follows from applying \autoref{thm:Menger} to  $G^{\sf R}$.

\section{Applications}\label{section:DDP-algorithm}

In this section we show how to exploit the duality between \Rpaths and \Rcut{s} to improve on results by Edwards et al.~\cite{Edwards2017} and Giannopoulou et al.~\cite{GiannopoulouKKK22}.
The following is the main result that we prove, and then we use it to improve on results by~\cite{Edwards2017,GiannopoulouKKK22}.
\begin{theorem}\label{theorem:solution_or_small_separator}
Let $k,c$ be integers with $k,c \geq 2$ and $g(k,c) = 2k(c\cdot k - c + 2) + c(k-1)$.
Let $G$ be a digraph, assume that we are given the bags of a bramble $\mathcal{B}$ of congestion $c$ and size at least $g(k,c)$, and $S,T \subseteq V(G)$ with $S = \{s_1, \ldots, s_k\}$ and $T = \{t_1, \ldots, t_k\}$.
Then in time $\Ocal(k^4 \cdot n^2)$ one can either
\begin{enumerate}
  \item find a $\mathcal{B}^* \subseteq \mathcal{B}$ with $|\mathcal{B}^*| \geq g(k,c) - c(k-1)$ and an $(S, \cupall{\mathcal{B}^*})$-separator $X_S$ with $|X_S| \leq k-1$ that is disjoint from all bags of $\mathcal{B}^*$, or
  \item find a $\mathcal{B}^* \subseteq \mathcal{B}$ with $|\mathcal{B}^*| \geq g(k,c) - c(k-1)$ and an $(\cupall{\mathcal{B}^*}, T)$-separator $X_T$ with $|X_T| \leq k-1$ that is disjoint from all bags of $\mathcal{B}^*$, or
  \item find a set of paths $\{P_1, \ldots, P_k\}$ in $G$ such that each $P_i$ with $i \in [k]$ is a path from $s_i$ to $t_i$ and each vertex of $G$ appears in at most $c$ of these paths.
\end{enumerate}
\end{theorem}

\autoref{theorem:solution_or_small_separator} yields an {\XP} algorithm with parameter $k$ for the $(k,c)$-\textsc{DDP} problem in $k$-strong digraphs, as we proceed to discuss.
First, we remark that the {\XP} time is only required when a large bramble of congestion at most $c$ is not provided.
If this is the case, we first look at the directed tree-width of $G$, which can be approximated in {\FPT} time applying \autoref{prop:FPT-dtw-or-bramble}.
If $\dtw(G) \leq f(k)$ for some computable function $f$, then we solve the problem applying \autoref{proposition:XP-algo-DDP-congestion}.
Otherwise, we apply the machinery by Edwards et al.~\cite{Edwards2017} pipelining \cref{prop:well-linked-set-and-path,prop:finding-bramble-congestion-two} to obtain a large bramble of congestion two in digraphs of large directed tree-width, and use \autoref{theorem:solution_or_small_separator} to find a solution in polynomial time, which we show to always be possible.
When assuming that the input digraph is $k$-strong, only the third output of \autoref{theorem:solution_or_small_separator} is possible, and therefore as a direct consequence of \autoref{theorem:solution_or_small_separator} we obtain the following.

\begin{theorem}\label{theo:always-solution-given-bramble}
Let $G$ be a $k$-strong digraph and $\mathcal{B}$ be a bramble of congestion $c \geq 2$ with $|\mathcal{B}| \geq 2k(c\cdot k - c + 2) + c(k-1)$.
Then for any ordered sets $S,T\subseteq V(G)$ both of size $k$, the instance $(G, S, T)$ of $(k,c)$-{\sc DDP} with $c \geq 2$ is positive and a solution can be found in time $\Ocal(k^4 \cdot n^2)$.
\end{theorem}

As mentioned in the introduction, our result improves over the result of Edwards et al.~\cite{Edwards2017} by relaxing the strong connectivity of the input digraph from $36k^3 + 2k$ to $k$ (and this bound is close to the best possible unless {\P} = {\NP}), by needing a smaller bramble (from size $188k^3$ to $4k^2 + 2k -2$ when $c=2$), and because the proof is simpler and shorter.
Finally, applying \autoref{cor:bounded-dtw-or-bramble-congestion-two}, \autoref{proposition:XP-algo-DDP-congestion}, and \autoref{theo:always-solution-given-bramble} (thus again using \autoref{prop:finding-bramble-congestion-two} by~\cite{Edwards2017}) we immediately obtain the following.
\begin{corollary}\label{cor:xp-algorithm-k-strong-digraphs}
For every integer $c \geq 2$, the $(k,c)$-{\sc DDP} problem is solvable in {\XP} time with parameter $k$ in $k$-strong digraphs.
\end{corollary}

As a tool to prove \autoref{theorem:solution_or_small_separator}, we  first show how we can take many copies of digraphs $G$ and $G'$, say $\ell$ copies of each, to either find $2\ell \cdot k$ \Rpaths to a given collection $\mathcal{B}$ of sufficiently large size, or find an appropriate separator of size at most $k - 1$ in $G$ or in $G'$.
\begin{lemma}\label{lem:R-paths-or-small-cut}
Let $(\mathcal{F, S})$ be a digraph-source sequence where $\mathcal{F}$ contains $\ell$ copies of a digraph $G_S$ and $\ell$ copies of a digraph $G_T$, in this order, and $\mathcal{S}$ contains $\ell$ copies of a set $S \subseteq V(G_S)$ and $\ell$ copies of a set $T \subseteq V(G_T)$, in this order, where $|S| = |T| = k$.
Finally, let $\mathcal{B}$ be a collection of subsets of $V(G)$ with $|\mathcal{B}| \geq 2\ell \cdot k$.

Then either there is a set of \Rpaths with respect to $\mathcal{F, S}$, and $\mathcal{B}$ of size at least $2\ell \cdot k$, or for some non-empty $\mathcal{B}' \subsetneq \mathcal{B}$ there is an \Rcut $(\mathcal{B', X})$ of order at most $2\ell\cdot k - 1$ such that $\mathcal{X}$ contains $\ell$ copies of an $(S, \cupall{\mathcal{B}'})$-separator $X_S$ followed by $\ell$ copies of an an $(T, \cupall{\mathcal{B}})$-separator $X_T$ with $|X_S| + |X_T| \leq 2k-1$.
\end{lemma}

\begin{proof}
Assume that the maximum size of a set of \Rpaths with respect to $\mathcal{F, S}$, and $\mathcal{B}$ is at most $2\ell \cdot k - 1$.
Then by \autoref{theo:min-max-statement-var-3} there is a minimum \Rcut $Y = (\mathcal{B}', \mathcal{X}')$ of order at most $2\ell \cdot k - 1$.

Since $|\mathcal{B}| \geq 2\ell \cdot k$ we conclude that $\mathcal{B}' \neq \emptyset$ and thus $|\mathcal{X}'| = 2\ell$. 
(we refer the reader to the discussion in the end of the first part of \autoref{sec:newpaths}).
Hence by the choice of $\mathcal{F}$ and $\mathcal{S}$ and the definition of \Rcut{s}, we can construct an \Rcut $(\mathcal{B}', \mathcal{X})$ with the same order as $Y$ by simply including in $\mathcal{X}$ exactly $\ell$ copies of the $(S, \cupall{\mathcal{B}'})$-separator $X_S$ contained in $\mathcal{X}'$, and $\ell$ copies of the $(T, \cupall\mathcal{B}')$-separator $X_T$ contained in $\mathcal{X}'$.
Thus the newly generated \Rcut satisfies  $|\mathcal{B} \setminus \mathcal{B}'| + \ell(|X_S| + |X_T|) = \order(\mathcal{B', X}) = \order(\mathcal{B', X'})  \leq 2\ell k - 1$.
This immediately implies that $|X_S| + |X_T| \leq 2k-1$ and the result follows.
\end{proof}

The following definition is used in the proof of \autoref{theorem:solution_or_small_separator}.
Since \Rpaths only consider paths leaving the sets $S_i$ and arriving in the elements of $\mathcal{B}$, and the goal is to find $S \to T$ paths, we construct a digraph-source sequence by taking many copies of the digraph $G$, and many copies of the digraph resulting from \emph{reversing} the orientation of every edge of $G$.
\begin{definition}[Reverse digraph]\label{def:reverse-digraph}
Let $G$ be a digraph.
The \emph{reverse digraph} $G^{\sf rev}$ of $G$ is the digraph constructed by reversing the orientation of every edge of $G$.
That is, $(u,v) \in E(G^{\sf rev})$ if and only if $(v,u) \in E(G)$.
\end{definition}

Now the plan is to apply twice the duality between \Rpaths and \Rcut{s}.
In the first application, we consider the digraph-source sequence formed by $2k(c\cdot k - c + 1)$ copies of $G_S = G$ and then exactly as many copies of $G_T = G^{\sf rev}$, and the same number of copies of $S$ and $T$, also in this order.
If we do not find a set of $2k(c\cdot k - c + 1)$ \Rpaths to the bags of the bramble, then we apply \autoref{lem:R-paths-or-small-cut} and find an \Rcut $(\mathcal{B', X})$ of small order and a separator of size at most $k - 1$ in $G_S$ or $G_T$.
Since the given bramble has congestion $c$, we show that we can stop with output \lipItem{1} or \lipItem{2} of \autoref{theorem:solution_or_small_separator} by taking the bramble $\mathcal{B}^*$ containing all bags of $\mathcal{B}$ that are disjoint from the small separator.
This holds because no bag of $\mathcal{B}^*$ can be in the ``wrong side'' of the separator.
That is, if for instance there is an $S \to A$ path $P$ avoiding the separator $X_S$ of size at most $k-1$ in $G_S$, then every bag of $\mathcal{B}^*$ must be in $\mathcal{B} \setminus \mathcal{B}'$ since otherwise one can construct a path in $G_S \setminus X_S$ from $S$ to a bag of $\mathcal{B}'$ by using $P$ and the connectivity properties of the bramble.
In this case, we also arrive at a contradiction since the size of $\mathcal{B}^*$ is too large to be contained in the \Rcut.

If the \Rpaths are found, then we refine the digraph by carefully choosing edges to delete from $G$ in such a way that, from a second application of the duality between \Rpaths and \Rcut{s}, we are guaranteed to find $2k$ \Rpaths that can be used to construct the $\{s_i\}\to\{t_i\}$ paths while maintaining the congestion under control.
In the refined digraph, we keep the \Rpaths found in the first iteration and delete edges leaving vertices of the bramble appearing in bags that were not used as destinations for the \Rpaths.
We apply \autoref{lem:R-paths-or-small-cut} in this digraph and show that the only possible outcome is that the \Rpaths are found.
Otherwise, there is an \Rcut of order at most $2k-1$ and hence there is a $\mathcal{B}'' \subseteq \mathcal{B}$ that contains at least one bag $A$ that is not in the \Rcut and is disjoint from the separator, say $X'_S$, of size at most $k-1$ that is part of the \Rcut.
Now the size of $X'_S$ implies that there is a path from $S$ to a bag disjoint from $X'_S$ avoiding the separator, and thus we can reach $A$ from $S$ avoiding $X'_S$.
Finally, the refined digraph allows us to associate each vertex of the bramble used by a path in $\{P_1, \ldots, P_k\}$ with a bag of the bramble, depending on where the vertex appears in the path, in such a way that no vertex is associated twice with the same bag by two distinct paths.
Together with the bound on the congestion of the bramble, this immediately implies that every vertex of the digraph appears in at most $c$ paths of the set $\{P_1, \ldots, P_k\}$.
The details follow.
\begin{proof}[Proof of \autoref{theorem:solution_or_small_separator}]
Let $G$, $S$, $T$, and $\mathcal{B}$ be as in the statement of the theorem.
Define $G_S = G$ and $G_T = G^{\sf rev}$ and let $(\mathcal{F, S})$ be a digraph-source sequence with $\mathcal{F}$ containing, in order, $c \cdot k - c + 1$ copies of $G_S$ followed by $c \cdot k - c + 1$ copies of $G_T$, and $\mathcal{S}$ containing, in order, the same number of copies of $S$ followed by exactly as many copies of $T$.
Now, applying  \autoref{lem:R-paths-or-small-cut} with respect to $\mathcal{F, S, B}$, and $\ell = c\cdot k - c + 1$, we conclude that either there are $2k(c\cdot k - c + 1)$ \Rpaths or, for some $\mathcal{B}' \subseteq \mathcal{B}$ there is an \Rcut $(\mathcal{B', X})$ where $\mathcal{X}$ contains $\ell$ copies of an $(S, \cupall\mathcal{B}')$-separator $X_S$ and $\ell$ copies of an an $(S, \cupall\mathcal{B}')$-separator $X_T$ with $|X_S| + |X_T| \leq 2k-1$.
We first consider the case where the separators are obtained.
Thus $|X_S| \leq k-1$ or $|X_T| \leq k-1$.
Since both cases are symmetric (the $T \to X_T$ paths become $X_T \to T$ paths when we restore the orientation of the edges of $G_T$), we suppose without loss of generality that  $|X_S| \leq k-1$.

Let $\mathcal{B}^*$ contain all bags of $\mathcal{B}$ that are disjoint from $X_S$.
Since $\mathcal{B}$ has congestion $c$ we conclude that $|\mathcal{B}^*| \geq g(k,c) - c(k - 1)$.
We show that no vertex appearing in a bag of $\mathcal{B}^*$ is reachable from $S$ in $G_S \setminus X_S$.
By contradiction, assume that there is an $S \to A$ path $P$ in $G_S \setminus X_S$ for some $A \in \mathcal{B}^*$.
If there is $A' \in \mathcal{B}^* \cap \mathcal{B}'$ then we can use the strong connectivity of $G_S[A \cup A']$ and the path $P$ to construct an $S \to A'$ path in $G_S \setminus X_S$, contradicting the choice of $X_S$.
Thus in this case $\mathcal{B}^*$ must be entirely contained in $\mathcal{B} \setminus \mathcal{B}'$.
Again we obtain a contradiction since $2k(c\cdot k - c + 2) \leq |\mathcal{B^*}| \leq |\mathcal{B} \setminus \mathcal{B}'| \leq \order(\mathcal{B', X}) < 2k(c\cdot k - c + 1)$.
In other words, the existence of path from $S$ to a vertex in a bag of $\mathcal{B}^*$ avoiding $X_S$ implies that every bag of $\mathcal{B}^*$ is reachable from $S$ in $G_S \setminus X_S$ and thus such path cannot exist since $\mathcal{B}^*$ is too large to be contained in any \Rcut of order less than $2k(c \cdot k - c + 1)$.
We conclude that no bag $A \in \mathcal{B}^*$ is reachable from $S$ in $G_S \setminus X_S$, and and output \lipItem{1} of the theorem follows.
Symmetrically, output \lipItem{2} of the theorem follows if $|X_T| \leq k-1$.

Assume now that a set of \Rpaths $\mathscr{P}$ of size at least $2k(c\cdot k - c + 1)$ is found.
Let $\mathcal{B}^1 \subseteq \mathcal{B}$ and let $h: \mathscr{P} \to \mathcal{B}^1$ be the bijective mapping as in \lipItem{(1c)} of \autoref{def:r-paths-and-r-cuts}.
Thus $h(p) = B$ implies that $\sink(P) \in B$.
Since $|S| = |T| = k$, we can split $\mathscr{P}$ into two sets of equal size $\mathscr{P}^S$ and $\mathscr{P}^T$, where every path in $\mathscr{P}^S$ is a path leaving $S$ in $G_S$ and every path in $\mathscr{P}^T$ is a path leaving $T$ in $G_T$.
We define $\mathcal{B}_S = \{h(P) \in \mathcal{B}^1 \mid P \in \mathscr{P}^S\}$  and $\mathcal{B}_T = \{h(P) \in \mathcal{B}^1 \mid P \in \mathscr{P}^T\}$.
The next step is to refine $\mathscr{P}$ to make it {\sl minimal with respect to $\mathcal{B}$}.
That is, we say that $\mathscr{P}$ is \emph{$\mathcal{B}$-minimal} if no path of $\mathscr{P}$ contains an internal vertex that is in a bag $\mathcal{B} \setminus (\mathcal{B}_S \cup \mathcal{B}_T)$.
In other words, when following a path $P \in \mathscr{P}$ from the first to the last vertex, if we find an internal vertex $v$ in a bag $B$ that is not in $\mathcal{B}_S$ nor in $\mathcal{B}_T$, we swap $P$ in $\mathscr{P}$ by its subpath ending in $v$ and update $\mathcal{B}_S$ or $\mathcal{B}_T$ accordingly.
Clearly, condition \lipItem{(1c)} of \autoref{def:r-paths-and-r-cuts} still hold with respect to the new choice of $\mathscr{P}$, and therefore from now on we assume that $\mathscr{P}$ is $\mathcal{B}$-minimal.
This property is important later to bound the maximum number of times that a vertex can appear in the $S \to T$ paths we construct.

Next, we reduce $G_S$ and $G_T$ to digraphs that are still well connected to $\mathcal{B}$, thus ensuring that we can find a large set of \Rpaths in the new digraphs, in which we can maintain control over how many times a vertex can appear in the $S \to T$ paths we construct from these new \Rpaths.
To this end, define
\begin{equation}\label{eq:sets-VS-and-VT}
V_S = \bigcup_{P \in \mathscr{P}^S}(V(P) \setminus \{\sink(P)\}) \cup \bigcup_{A\in \mathcal{B}_S}A \  \text{ and } \ V_T = \bigcup_{P \in \mathscr{P}^T}(V(P) \setminus \{\sink(P)\}) \cup \bigcup_{A\in \mathcal{B}_T}A.
\end{equation}
Finally, we construct the digraphs $G'_S$ and $G'_T$ starting from $G_S$ and $G_T$, respectively, applying the following two rules:

\begin{itemize}
\item For every $v \in V(G_S)$, if  $v \in (\cupall\mathcal{B}) \setminus V_S$ then we delete from $G'_S$ every edge leaving $v$.
\item For every $v \in V(G_T)$, if  $v \in (\cupall\mathcal{B}) \setminus V_T$ then we delete from $G'_T$ every edge leaving $v$.
\end{itemize}

Consider the digraph-source sequence $(\{G'_S, G'_T\}, \{S, T\})$ and let $\mathcal{B}' = \mathcal{B} \setminus (\mathcal{B}_S \cup \mathcal{B}_T)$, and notice that $\mathcal{B}'$ may not be a bramble in $G'_S$ nor in $G'_T$.
Clearly $|\mathcal{B}'| \geq g(k,c) - 2k(c \cdot k -c + 1) = c(k-1) + 2k > 2k$. 
We apply \autoref{lem:R-paths-or-small-cut} with respect to $\{G'_S, G'_T\}, \{S, T\}$ (and thus $\ell = 1$), and $\mathcal{B}'$, to either obtain a set of \Rpaths $\mathscr{P}'$ of size at least $2k$ or an \Rcut $(\mathcal{B}'', \{X'_S, X'_T\})$ with order at most $2k-1$ where $|X'_S| + |X'_T| \leq 2k-1$ and $\mathcal{B}''\neq \emptyset$.
We claim that only the first output is possible.
By contradiction, assume that the \Rcut and the separators were obtained and, without loss of generality, that $|X_S| \leq k-1$.
First notice that the upper bound on the order of the \Rcut implies that $|\mathcal{B}'' \setminus \mathcal{B}'| \geq c(k-1) + 1$.
Since $|X_S| \leq k-1$ this implies that there is at least one bag $A' \in \mathcal{B}''$ that is disjoint from $X_S$ and not included in \Rcut.

Now, set $q = 2(c \cdot k - c + 1)$ and let $\mathcal{P}_1, \ldots, \mathcal{P}_q$ be the defining partition of~$\mathscr{P}$.
That is, for every $i \in [q]$, the part $\mathcal{P}_i$ is a set of $k$ disjoint paths in the $i$-th digraph of $\mathcal{F}$ (this is possible since $|S| =|T| = k$ and hence $\mathscr{P}$ cannot contain more than $k$ disjoint paths in any digraph in $\mathcal{F}$).
Thus exactly $q/2$ parts $\mathcal{P}_i$ contain only paths starting in $S$.
Now the size of $X_S$ allows it to intersect at most $c(k-1)$ bags of $\mathcal{B}$, and thus for some $\mathcal{P}_i$ no bag in $\mathcal{B}_i = \{A \in \mathcal{B}_S \mid P \in \mathcal{P}_i \text{ and }\sink(P) \in A\}$ is intersected by $X'_S$.
Since $|X'_S| \leq k-1$, there is a $P \in \mathcal{P}_i$ from $S$ to a bag $A \in \mathcal{B}_i$ that is not intersected by $X'_S$.
By the choice of $G'_S$ this path also exists in this digraph and, since $\sink(P) \in A$ and $A \in \mathcal{B}_S$, every edge of $G_S$ leaving every vertex in $A$ is kept in $G'_S$ and thus $G'_S[A]$ is strong.
Now, as $\mathcal{B}$ is a bramble, we can construct a path from $S$ to $A'$ (remember that $A'\in \mathcal{B}''$ and $A' \cap X_S = \emptyset$) by following the path $P$ and then taking a path from $\sink(P)$ to $A'$ in $G'_S[A \cup A']$, which in turn is guaranteed to exist since either $A \cap A' \neq \emptyset$ or there is an edge from $A$ to $A'$ in $G'_S$.
This contradicts our assumption that $(\mathcal{B}'', \{X'_S, X'_T\})$ is an \Rcut and the claim follows.

Assume now that a set $\mathscr{P}'$ of $2k$ \Rpaths is obtained.
Therefore, there are disjoint paths $\{Q^S_1, \ldots, Q^S_k\}$ leaving $S$ and disjoint paths $\{Q'_1, \ldots, Q'_k\}$ leaving $T$ in $\mathscr{P}$.
By recovering the orientation of the edges of $G'_T$ (we remind the reader that $G_T = G^{\sf rev}$), we construct paths $\{Q^T_1, \ldots, Q^T_k\}$ {\sl reaching} $T$ in $G$ and, by renaming the paths if needed, we assume that, for $i \in [k]$, each $Q^S_i$ is a path starting in $s_i$ and each $Q^T_i$ is a path ending in $t_i$.
Moreover, by condition~\lipItem{(1c)} in the definition of \Rpaths (see \autoref{def:r-paths-and-r-cuts}), each $\sink(Q^S_i)$ is associated with a unique bag $A_i \in \mathcal{B}$ and each $\source(Q^T_i)$ is associated with a unique bag $A'_i \in \mathcal{B}$, such that all bags $A_1, \ldots, A_k, A'_1, \ldots, A'_k$ are distinct.
Hence, since $\mathcal{B}$ is a bramble, it follows that for every $i \in [k]$ we can find a shortest path $Q_i$ from $\sink(Q^S_i)$ to $\source(Q^T_i)$ in the strong digraph $G[A_i \cup A'_i]$.
Finally, we construct the desired paths $\{P_1, \ldots, P_k\}$, such that each $P_i$ with $i \in [k]$ is a path from $s_i$ to $t_i$, by appending $Q_i$ to $Q^S_i$ and then $Q^T_i$ to the resulting path.
Notice that this construction may result in a walk instead of a path, but every walk can be easily shortened into a path $P_i$.

We now claim that every vertex of $G$ appears in at most $c$ paths of the collection $\{P_1, \ldots, P_k\}$.
First, notice that since the paths $\{Q^S_1, \ldots, Q^S_k\}$ are disjoint and the paths $\{Q^T_1, \ldots, Q^T_k\}$ are disjoint as well, any vertex not appearing in any bag of $\mathcal{B}$ can appear in at most two paths of $\{P_1, \ldots, P_k\}$.
Assume now that $v$ is a vertex appearing in some bag of $\mathcal{B}$.
Depending on where $v$ is located in the paths $Q^S_i$, $Q^T_i$, and $Q_i$, we associate $v$ with a bag of $\mathcal{B}$.
Since $\mathcal{B}$ has congestion $c$, this immediately validates the claim and the result follows.
We remind the reader of our assumption that $\mathscr{P}$ is $\mathcal{B}$-minimal, and look again at \autoref{eq:sets-VS-and-VT}.
If $v$ is in $V_S$ because $v$ is in path $P \in \mathscr{P}^S \cup \mathscr{P}^T$ and $v \neq \sink(P)$, then we say that $v$ is a \emph{type $1$} vertex.
Otherwise, we say that $v$ is a \emph{type $2$} vertex.

For $i \in [k]$, if $v$ is a internal vertex of some $Q^S_i$, then $v \in V^S$ and is either a type $1$ or a type $2$ vertex.
If $v$ is of type $1$, then $v$ is an internal vertex of some path $P \in \mathscr{P}^S$ and $v \neq \sink(P)$.
Since $\mathcal{P}$ is $\mathcal{B}$-minimal, this implies that $v$ is in some destination bag of $\mathcal{B}_S$ and we associate $v$ with this bag.
If $v$ is of type $2$, then $v$ is not an internal vertex of any path in $\mathscr{P}^S$ and is in some bag of $\mathcal{B}_S$.
We associate $v$ with this bag.
Since the paths $\{Q^S_1, \ldots, Q^S_k\}$ are disjoint, $v$ appears only in one of those paths and thus no other $Q^S_j$ can associate $v$ with another bag of $\mathcal{B}$.
The analysis is similar if $v$ is an internal vertex of some $Q^T_j$.
Notice that it is possible that $v$ is in both $Q^S_i$ and $Q^T_j$ and, in this case, those two paths associate $v$ with two distinct bags of $\mathcal{B}_S$ and $\mathcal{B}_T$, respectively.

Now let $\mathcal{B}^2 \subseteq \mathcal{B}'$ and let $h': \mathscr{P} \to \mathcal{B}^2$ be a bijective mapping as in \lipItem{(1c)} of \autoref{def:r-paths-and-r-cuts}.
If $v$ is a vertex of some $P_i$ from $\sink(Q^S_i)$ to $\source(Q^T_i)$ then we associate $v$ with $h'(Q^S_i)$ if $v \in h'(Q^S_i)$, and we associated $v$ with $h'(Q^T_i)$ if $v \in h'(Q^T_i)$.

Now, for $i,j \in [k]$, every path of the form $Q^S_i$, $Q^T_i$, or $P_i$ associates each of its vertex inside of the bramble with a unique bag of $\mathcal{B}$, each vertex associated with some bag appears in $V(Q^S_i)\setminus \{\sink(Q^S_i)\}$, $V(Q^T_i) \setminus \{\source(Q^T_i)\}$, or $V(Q_i)$, no two distinct $Q^S_i, Q^T_j$ associate a vertex $v$ with the same bag, and the same holds with relation to distinct pairs of paths of the form $Q_i, Q_j$ and $Q^T_i, Q^T_j$.
We remark that while it is possible that some $v$ appears in both $Q^S_i$ and $Q^T_i$, this does not pose an issue since in this case $v$ is associated with a pair of distinct bags $B \in \mathcal{B}_S$ and $B' \in \mathcal{B}_T$ by $Q^S_i$ and $Q^T_i$, respectively.
Since $\mathcal{B}$ has congestion $c$, it follows that every vertex is associated with at most $c$ bags, which implies that every vertex is in at most $c$ paths of $\{P_1, \ldots, P_k\}$, and the result follows.

The bound on the running time follows by \autoref{theo:min-max-statement-var-3} and by observing that a set of \Rpaths can be made $\mathcal{B}$-minimal in time $\Ocal(c\cdot k^2 \cdot n^2)$.
\end{proof}

Next, we discuss how to apply \autoref{theorem:solution_or_small_separator} in the context of the asymmetric version of $(k,c)$-\textsc{DDP} as an improvement over~\cite[Theorem 9.1]{Giannopoulou2020}.

\subsection{Application to the asymmetric version of $(k,c)$-\textsc{DDP}}\label{subsection:application-assymetric-version}
\label{section:asymmetric}

\autoref{theorem:solution_or_small_separator} is a direct translation of Giannopoulou et al.~\cite[Theorem 9.1]{Giannopoulou2020} to our setting.
As mentioned in the introduction, we can prove a weaker version of \autoref{theo:always-solution-given-bramble} by replacing \autoref{theorem:solution_or_small_separator} by~\cite[Theorem 9.1]{Giannopoulou2020}.
In comparison with the result by Edwards et al.~\cite{Edwards2017}, with this approach we can drop the bound on the strong connectivity of the digraph from $(36k^3 + 2k)$ to $k$, and the trade-off is that in this case we have to rely on the topology of {\sl cylindrical grids} to connect the paths, instead of a bramble of congestion $c$.
Although it is true that every digraph with large directed tree-width contains a large cylindrical grid~\cite{KawarabayashiK15}, and that such a grid can be found in {\FPT} time~\cite{Campos2022}, to find a cylindrical grid of order $k$ the directed tree-width of the digraph has to increase much more than it is needed to find a bramble of congestion two (although both dependencies still consist of a non-elementary tower of exponentials).
Additionally, we remark that if the goal is to solve the $(k, c)$-\textsc{DDP} problem with $c \geq 8$, then as stated in \autoref{theo:always-solution-given-bramble} a bramble of congestion eight suffices, and a polynomial dependency on how large the directed tree-width of a digraph must be to guarantee the existence of a large bramble of congestion eight was shown by Masarík et al.~\cite{MasarikPRS22}.

On the other hand, \autoref{theorem:solution_or_small_separator} improves upon~\cite[Theorem 9.1]{Giannopoulou2020} in both its statement, since we use brambles instead of cylindrical grids, and in simplicity.
Indeed, brambles of bounded congestion seem to be a weaker structure than cylindrical grids, since it possible to extract such brambles with order $t$ from a cylindrical grid of order at least $2t$ (see \cite[Lemma 9]{Edwards2017}), and the bound on how large the directed tree-width of a digraph has to be to guarantee the existence of such a bramble with size $t$ is, in many cases (see~\cite{MasarikPRS22} for instance) and as far as we know also in the general case, substantially better than what is needed to find a cylindrical grid with the same order.
Additionally, the algorithm to find cylindrical grids runs in {\FPT} time (see~\cite{Campos2022}) provided that a certificate of large directed tree-width is given and, in contrast, a large bramble of congestion two can be found in polynomial time when such certificates are provided, as stated in \autoref{prop:finding-bramble-congestion-two}.
Finally, we only ask the bramble to have order $2k(c\cdot k - c + 2)+ c(k-1)$ (which corresponds to $4k^2 + 2k - 2$ when $c=2$) instead of the requested order $k(6k^2 + 2k + 3)$ for the cylindrical grid in the statement of~\cite[Theorem 9.1]{Giannopoulou2020}, where the goal is to compute solutions for $(k,2)$-\textsc{DDP}.
Their proof relies on the topology of cylindrical grids to connect the paths inside of this structure, after some careful selection on how to reach it from $S$ and leave it to reach $T$.
In our proof of \autoref{theorem:solution_or_small_separator}, it is very simple to connect the paths inside the bramble.
Indeed, after applying twice the duality between \Rpaths and \Rcut{s}, for each $i \in [k]$ we simply connect the ending vertex of the path from $s_i$ to the bramble containing the starting vertex of the path from the bramble to $t_i$, using the strong connectivity of the digraph induced by $B\cup B'$, where $B$ is the bag associated with $s_i$ and $B'$ the bag associated with $t_i$.

Their result~\cite[Theorem 9.1]{Giannopoulou2020} is one of the cornerstones in their algorithm to solve the asymmetric version of $(k,2)$-\textsc{DDP}.
Recall that, given ordered sets of terminals $\{s_1, \ldots, s_k\}$ and $\{t_1, \ldots, t_k\}$, the goal in this asymmetric version is to either produce a collection of paths from each $s_i$ to the corresponding $t_i$ such that every vertex is in at most two paths of the collection, or conclude that there is no collection of disjoint $\{s_i\} \to \{t_i\}$ paths.
In the first case, we say that we have constructed a \emph{half-integral linkage}. In the second case, we say that we have a {\sf no}-instance.
At any point of their dynamic programming algorithm, if one of the subproblems they define deals with an instance in which
there is no small separator intersecting all paths from $S$ to the grid or from the grid to $T$, then they apply~\cite[Theorem 9.1]{Giannopoulou2020} to find a solution to the instance.
If a separator is found, then they generate two easier instances, one of bounded directed tree-width, and one with fewer number of terminals.
Intuitively, the same holds true if we substitute~\cite[Theorem 9.1]{Giannopoulou2020} by our \autoref{theorem:solution_or_small_separator}, and it is not hard to write a formal proof verifying this.
We give an informal intuition of why this is the case below.

Let $(G, S, T)$ be an instance of the asymmetric version of $(k,2)$-\textsc{DDP}, with ordered sets $S,T \subseteq V(G)$ as described in the previous paragraph.
The goal is to adapt~\cite[Theorem 10.1]{Giannopoulou2020} using \autoref{theorem:solution_or_small_separator}.
In short, the point is that in the proof of the former result by Giannopoulou et al.~\cite{Giannopoulou2020}, the actual topology of the cylindrical grid is not relevant.
The important observations are that (i) every digraph not containing a cylindrical grid of order $t$ has directed tree-width bounded by some function of $t$, and (ii) if there are many disjoint paths from $S$ to the cylindrical grid and from the grid to $T$, then an appropriate half-integral linkage can be found in polynomial time.
To apply their strategy, we have to show that the following (informal) statement holds.

``\emph{For any choice of $(G, S, T)$ where $G$ is a digraph where no two large brambles of congestion two (whose sizes depend on $g(k,2)$ as in \autoref{theorem:solution_or_small_separator}) are separated by a small  separator (whose size depends on $k$), then the problem on $(G, S, T)$ can be solved in {\XP} time.}''

We can assume that  $G$ contains a large bramble $\mathcal{B}$ of congestion two, since otherwise we can solve $(G,S,T)$ applying \autoref{cor:bounded-dtw-or-bramble-congestion-two} and \autoref{proposition:XP-algo-DDP-congestion}.
Following the strategy of~\cite{Giannopoulou2020}, we apply induction on the size $k$ of $|S|$ and $|T|$, observing that at each step of the induction a subbramble of $G$ with size linear in $k$ is lost and thus at the starting point the requested order for $\mathcal{B}$ is larger than $g(k,2)$.

We apply \autoref{theorem:solution_or_small_separator} with $c=2$ and, if no separators are found, we output a solution for $(G, S, T)$.
Otherwise one of the separators, say $X_S$ (the other case is symmetric), is obtained.
Denote by $G^{\sf left}$ the digraph induced by $X_S$ plus all vertices that are reachable from $S$ in $G \setminus X_S$, and by $G^{\sf right}$ the digraph induced by $(V(G) \setminus {V(G^{\sf left})}) \cup X_S$.
For simplicity, assume that $T \subseteq V(G^{\sf right}) \setminus V(G^{\sf left})$.
By our assumption that output \lipItem{(1)} of \autoref{theorem:solution_or_small_separator} was obtained, $G^{\sf right}$ contains a large bramble $\mathcal{B}'$ that is separated from $S$ by $X_S$.

Since the goal is to obtain an {\XP} algorithm, we are allowed to guess how a solution can cross the separator $X_S$.
That is, we guess the ordered sets
\begin{itemize}
  \item $S' \subseteq V(G^{\sf left})\setminus X_S$,
  \item $X^{\sf mid} \subseteq X_S$, and
  \item $T' \subseteq V(G^{\sf right}) \setminus X_S$,
\end{itemize}
and the goal is to solve instances of $(k,2)$-\textsc{DDP} of the form $(G^{\sf left}, S', X^{\sf mid})$ and $(G^{\sf right}, X^{\sf mid}, T')$ (we omit some details on the choices of the new instances in order to shorten this informal discussion).
A simple enumeration argument suffices to show that the number of instances is bounded by a function depending on $k$.

By our assumption that no two large brambles of congestion two are separated by a small separator, we conclude that $G^{\sf left}$ cannot contain a large bramble of congestion two and thus the directed tree-width of $G^{\sf left}$ is bounded by some function of $k$.
Thus, we can apply \autoref{proposition:XP-algo-DDP} to solve instances of $k$-\textsc{DDP} in $G^{\sf left}$. 
For instances on $G^{\sf right}$, we apply the induction hypothesis.
We remind the reader that our assumption that output \lipItem{1} of \autoref{theorem:solution_or_small_separator} was obtained implies that $G^{\sf right}$ contains a large bramble of congestion two.
The fact that $|X_S| \leq k-1$  immediately implies that $|X^{\sf mid}| \leq k-1$, and it remains to argue that no two large (where the size now depends on $(g(k-1, 2)$) brambles of congestion two in $G^{\sf left}$ are separated by a small separator (whose size now depends on $k-1$).
This can be  done by a counting argument as in the proof of~\cite[Theorem 10.2]{Giannopoulou2020}.

If for some choice of $X^{\sf mid}$ we obtain a solution $\mathscr{P}^{\sf left}$ for an instance $(G^{\sf left}, S', X^{\sf mid})$ of $k$-\textsc{DDP} and a solution $\mathscr{P}^{\sf right}$ for the instance $(G^{\sf right}, X^{\sf mid}, T')$, then we can glue together both solutions to construct a half-integral solution for $(G, S, T)$.
If this is not the case, then we conclude that $(G, S, T)$ is a {\sf no}-instance.

Therefore, with a new version of~\cite[Theorem 10.1]{Giannopoulou2020} it is straightforward to apply it together with the decomposition by Giannopoulou et al.~\cite{Giannopoulou2020} to obtain a new proof of an {\XP} algorithm for the asymmetric version of $(k, 2)$-\textsc{DDP}.


\bibliography{main}
\end{document}